\documentclass{amsart}
 \usepackage{amsmath}
\usepackage{amssymb}
\usepackage{amsfonts}
 \usepackage{eucal}
      \makeatletter
     \def\section{\@startsection{section}{1}%
     \z@{.7\linespacing\@plus\linespacing}{.5\linespacing}%
     {\bfseries
     \centering
     }}
     \def\@secnumfont{\bfseries}
     \makeatother
\usepackage{a4wide}
\newtheorem{theorem}{Theorem}[section]
\newtheorem{lemma}[theorem]{Lemma}
\newtheorem{corollary}[theorem]{Corollary}
\newtheorem{proposition}[theorem]{Proposition}
\theoremstyle{definition}

\theoremstyle{remark}
\newtheorem{remark}[theorem]{Remark}
\theoremstyle{remarks}

\numberwithin{equation}{section}

\def \a{{\alpha}}
\def \b{{\beta}}
\def \D{{\Delta}}

\def \d{{\delta}}
\def \e{{\varepsilon}}

\def \g{{\gamma}}

\def \k{{\kappa}}

\def \p{{\varphi}}
\def \t{{\vartheta}}

\def \m{{\mu}}
\def \s{{\sigma}}

\def \P{{\bf P}}

\def \qq{{\qquad}}
\def \R{{\bf R}}

 \def \Z{{\bf Z}}
 \def \dd{{\rm d}}

\def \noi{{\noindent}}

\def\E{{\mathbb E \,}}

\def\P{{\mathbb P}}

\def\R{{\mathbb R}}

\def\Z{{\mathbb Z}}

at  8 pt
\scrollmode

\font\gsec= cmb10 at 10 pt
   at 10,4  pt at  8,2 pt
   \scrollmode
\font\gsec= cmb10 at 10,2  pt

\font\sevenrm =cmr10 at  7  pt

    \title[Local Limit Theorem  with Effective Rate]
 {Approximate   Local Limit Theorems  with  Effective Rate and Application to Random Walks in Random Scenery}
  \author{
Rita   Giuliano  and    Michel   Weber}\address{IRMA, UMR 7501, Universit\'e
Louis-Pasteur,   7  rue Ren\'e Descartes, 67084
Strasbourg Cedex, France.
 \noi E-mail:    {\tt  michel.weber@math.unistra.fr
 }}
  \address{Dipartimento di Matematica, Via F. Buonarroti  2, 56127 Pisa, Italy.   E-mail: {\tt  giuliano@dm.unipi.it} }

\begin{document}

\maketitle



\def\ddate {\sevenrm \ifcase\month\or January\or
February\or March\or April\or May\or June\or July\or
August\or September\or October\or November\or December\fi\! {\the\day}, \!{\sevenrm\the\year}}


\renewcommand{\thefootnote}{} {{
\footnote{2010 \emph{Mathematics Subject Classification}: Primary: 60F15, 60G50 ; Secondary:
60F05.}
\footnote{\emph{Key words and phrases}: independent random variables,
lattice distributed,  Bernoulli part,  local limit theorem, effective remainder, random walk in random scenery.  \par  \sevenrm{[LLTR]5} \ddate{}}
 \renewcommand{\thefootnote}{\arabic{footnote}}
\setcounter{footnote}{0}


\begin{abstract} We show that the Bernoulli part extraction method  can be used to obtain approximate forms of the
local limit theorem  for sums of independent  lattice valued random
variables,  with effective error term, that is with explicit  parameters and universal
constants. We also show that our estimates allow to recover   Gnedenko and Gamkrelidze local limit theorems. We further establish by this method a local limit theorem with effective remainder for random walks in random scenery. \end{abstract}

\section{Introduction and Main Result.}
The        extraction method of the Bernoulli part  of a (lattice valued) random
variable  was developed by McDonald in
\cite{M},\cite{M1},\cite{MD} for proving   local limit theorems in presence of the central limit theorem. Twenty years before McDonald's
work,    Kolmogorov \cite{K}   initiated a similar approach  in the study of L\'evy's concentration function, and
is the first
having explored this direction.   For details and clarifications, we refer to    the recent paper by
Aizenmann, Germinet, Klein and Warzel
\cite{AGKW},   where this idea is also developed for general random variables and applications are given.

That method allows to transfer results which are available for systems of Bernoulli random variables to systems of arbitrary random
variables. It
  is based on a  probabilistic  device, and is proved to be an efficient alternative to the characteristic functions method.
Kolmogorov wrotes to this effect  in  his 1958's paper \cite{K}
p.29: \lq\lq... \!\!{\it Il semble cependant que nous restons toujours
dans une p\'eriode o\`u la comp\'etition de ces deux directions {\rm
[characteristic functions or direct methods from the calculus of
probability]} conduit aux r\'esultats les plus f\'econds ...\rq\rq}. We
believe that Kolmogorov's comment is still topical. \vskip 2pt  The
main object of this article is to show that this approach can be
used  to obtain, in a rather simple way, approximate forms of the
local limit theorem  for sums of independent  lattice valued random
variables,  with effective error term, that is with explicit  parameters and universal
constants.  The approximate form we obtain   expresses  quite
simply, and is thereby very handable. Further, it is   precise
enough to contain  Gnedenko and Gamkrelidze local limit theorems
(\ref{llt}). 
Before stating the
main results and in view of comparing results, it is necessary to  recall and discuss some classical facts and briefly
 describe the background of this problem.  
      Let $  \tilde
X=\{X_n , n\ge 1\}$ be a sequence of independent, square integrable
random variables taking values in a common lattice $\mathcal L(v_{
0},D )$ defined by the
 sequence $v_{ k}=v_{ 0}+D k$, $k\in \Z$, where
 $v_{0} $ and $D >0$ are   real numbers.   Let
\begin{equation}\label{not1}S_n=\sum_{j=1}^nX_j, \qq M_n=\sum_{j=1}^n\E X_j
 , \qq \Sigma_n=\sum_{j=1}^n{\rm Var}( X_j)
.
\end{equation}
Then $S_n$ takes values in the lattice
$\mathcal L( v_{ 0}n,D )$. The sequence $\tilde X$ satisfies a local limit theorem if
 \begin{equation}\label{llt}  \D_n:=  \sup_{N=v_0n+Dk }\Big|   \sqrt{\Sigma_n} \P\{S_n=N\}-{D\over  \sqrt{ 2\pi } }e^{-
{(N-M_n)^2\over  2 \Sigma_n} }\Big| = o(1).
\end{equation}

This is a fine limit theorem in Probability Theory, which also has
deep connections with Number Theory, see for instance Freiman  \cite{F} and  Postnikov  \cite{Po}. These two aspects of a same problematic
were  much studied in the past decades by the Russian School of probability. It seems however that some of these results are nowadays forgotten.
  \vskip 2pt    Assume that  $\tilde X$ is an i.i.d.  sequence and let $\m=\E X_1$, $\s^2={\rm Var}(
X_1)$. If for instance $X_1$ takes only even values, it is clear that (\ref{llt}) cannot be fulfilled with $D=1$.
 In fact,
  (\ref{llt}) holds  (with $M_n=n\m$, $\Sigma_n=n\s^2$)
if and only if the span $D$ is maximal,
 i.e.   there are no  other real numbers
$v'_{0}
$ and
$D' >D$ for which
$\P\{X
\in\mathcal L(v'_0,D')\}=1$. This  is   Gnedenko's well-known generalization
of the de Moivre-Laplace theorem.  Notice
that   (\ref{llt}) is significant only for the bounded domains of values
 \begin{equation}\label{lltrange}   {|N-n\m|} \le \s \sqrt{2n\log \frac{D }{ \e_n  }}   ,\end{equation}
where $\e_n\downarrow 0$ depends on the Landau symbol $o$.
         It is  worth observing that   (\ref{llt}) cannot be deduced from a central limit
theorem with rate, even under   stronger moment assumption.  Suppose for instance
$D=1$,
$X$ is centered and $\E|X|^3<\infty$. From Berry-Esseen's  estimate  only follows that
 $$\Big|\s  \sqrt{   n}  \P\{S_n=k\}-\s \sqrt n\int_{\frac{k}{\s \sqrt n}}^{\frac{k+1}{\s \sqrt
n}}e^{-t^2/2}\frac{\dd t}{\sqrt{2\pi}}\Big|\le C\frac{ \E|X|^3}{\s^2} .$$ However  the comparison term   has already the right order for
all integers $k$ such that $k+1\le \s \sqrt n$ since,
$$ \sup_{k+1\le \s \sqrt n}\big|\s \sqrt n\int_{\frac{k}{\s \sqrt n}}^{\frac{k+1}{\s \sqrt n}}e^{-t^2/2}\frac{\dd
t}{\sqrt{2\pi}}-\frac{1}{\sqrt{2\pi}}e^{-\frac{k^2}{2\s^2   n}}\big|\le \frac{C}{\s \sqrt n} \to 0. $$
  Hence   (\ref{llt}) cannot follow from it.
 Gnedenko' theorem is optimal: Matskyavichyus
\cite{Mat} showed that for any nonnegative function
$\p(n)\to 0$ as $n\to \infty$, there is an i.i.d.  sequence $\tilde X$, ($\E X_1=0$, $\E X_1^2<\infty$ and the form of the characteristic
function of $X_1$ is explicited)  such that for each $n\ge n_0$, $ \sqrt n \D_n\ge \p(n)$.
Stronger integrability properties yield better
remainder terms.
 \begin{theorem} \label{r}    Let
$F$ denote the distribution function of
$X$.\vskip 1pt \noi  {\rm  (i) (\cite{IL} Theorem
4.5.3)}      In order that  the property
\begin{equation} \label{alfa}   \sup_{N=an+Dk}\Big|  \sqrt n \P\{S_n=N\}-{D\over  \sqrt{ 2\pi}\s}e^{-
{(N-n\m )^2\over  2 n \s^2} }\Big| ={\mathcal O}\big(n^{-\alpha}\big) ,
\  0<\a<1/2 ,
 \end{equation}  holds,
 it is necessary and sufficient that the following conditions be satisfied:
 \begin{eqnarray*} (1) \   D \ \hbox{is maximal}, \qq \qq\qq
(2)  \  \  \int_{|x|\ge u} x^2 F(dx) = \mathcal O(u^{-2\a})\quad \hbox{as $u\to \infty$.}
\end {eqnarray*}
   {\rm (ii) (\cite{[P]} Theorem 6 p.197)} If $\E |X|^3<\infty$, then (\ref{alfa}) holds with $\a =1/2$.
\end{theorem}      The local
limit theorem in the independent case is often studied by using various structural characteristics, which are interrelated.  There exists a consequent literature (unfortunately no survey) and we only report a very few of them.  The "smoothness" characteristic
\begin{eqnarray}\label{delta}  \d_X =\sum_{k\in \Z}\big|\P\{X=v_k\}-\P\{X=v_{k+1}\}\big|,
\end{eqnarray}
   thoroughly investigated by Gamkrelidze is connected to the characteristic function $\p_X(t)=
 \E e^{i t X}$ through the relation
 \begin{eqnarray}\label{delta1a} (1-e^{it})\p_X(t)&=&\sum_{m\in \Z} \frac{(itm)}{m!}\big(\P\{X=m\}-\P\{X=m-1\}\big).
 \end{eqnarray}
Hence
 \begin{eqnarray}\label{delta1b} |\p_X(t)|&\le & \frac{\d_X}{2|\sin (t/2)|  }\qq\quad (t\notin 2\pi\Z).
 \end{eqnarray} This is used in Gamkrelidze \cite{G},   to prove the following well-known result:
  If a sequence
$\tilde X$ of independent integer valued random variables verifies:
\vskip 2pt   (i)   there exists an  $n_0$ such that
 $\sup
_k\ \d_{X_k^1+\ldots +X_k^{n_0}}\ <\ \sqrt 2$,  (here $X_k^{ j}, 1\le j\le n_0$ are independent copies of $X_k$),
\vskip 1pt   (ii)   the central
limit theorem is applicable,
\vskip 1pt   (iii)
${\rm Var}(S_n)=\mathcal O(n)$,
\vskip 2pt
 \noi then the local limit theorem is applicable in the strong form (i.e.
remains true when changing or discarding a finite number of terms of $\overline X$).

Later  Davis and McDonald
\cite{MD}    proved a variant of of  Gamkrelidze's result  using the Bernoulli part
extraction method.
       Let $X$ be a random variable such that
$\P \{X
\in {\mathcal L}(v_0,D) \}=1$,
  and let
\begin{eqnarray}\label{vartheta}  \t_X =\sum_{k\in \Z}\P\{X=v_k\}\wedge\P\{X=v_{k+1}\} ,
\end{eqnarray}
 where $a\wedge b=\min(a,b)$. Note  (section \ref{bercomp}) that
   necessarily $
\t_{X }<1$; moreover  
$
\d_X =2-2\t_X $ (Mukhin
 \cite{Mu}, p.700). This simple characteristic  is used in that method and it is required that $\t_X>0$.   More precisely
\begin{theorem} {\rm (\cite{MD},  Theorem 1.1)} Let $\{ X_j , j\ge 1\}$ be independent, integer valued random variables with partial sums
$S_n= X_1+\ldots +X_n$ and let
$f_j(k)=
\P\{X_j=k\}$.  Let also for each $j$ and $n$,
$$q(f_j)= \sum_{k} [f_j(k)\wedge f_j(k+1)], \qq Q_n=\sum_{j=1}^n q(f_j).  $$
Suppose that there are numbers $b_n>0$, $a_n$  such that
 $\lim_{n\to \infty}b_n= \infty$, $\limsup_{n\to \infty} {b_n^2}/{Q_n}<\infty$,   
and 
$$ \frac{S_n-a_n}{b_n} \ \ \buildrel{\mathcal L}\over{\Longrightarrow}\ \ \mathcal N(0,1).$$
  Then
$$ \lim_{n\to \infty} \sup_{k} \Big|b_n \P\{S_n=k\} -\frac{1}{\sqrt{2\pi}}e^{-  \frac{(k-a_n)^2}{2b_n^2}}\Big|=0. $$
 \end{theorem}

 \begin{remark} -- It may happen that $q(f_j)\equiv0$ and so $Q_n\equiv 0$. In   the above
(original) statement, it is thus implicitly assumed   that  $Q_n>0$,
$Q_n\uparrow\infty$ and $q(f_j)>0$, which is equivalent to
$f_j(k)\wedge f_j(k+1)>0$ for some $k\ge 0$.
 
 \noi --   It was recently shown in
Weber \cite{W}    that this method can also be   used efficiently to prove the almost sure local limit theorem in the critical case,
namely for sums of i.i.d.   random variables with the minimal integrability assumption: square integrability.
 \end{remark}  As   mentioned before, we are mainly interested  in   local limit
theorems with explicit   constants  in the remainder term.   There
are  generally speaking, much less related papers. Most  of the
local limit theorems with rate are  usually stated with Landau
symbols $o$, $\mathcal O$. And so the implicit constants may  depend
on the sequence   itself.
 Consider the   characteristic
$$ H(X ,d) = \E \langle X^*d\rangle^2,$$  where  $\langle \a \rangle$ is the distance from $\a$ to the nearest integer and  $X^*$
denotes  a symmetrization of
$X$.
   In Mukhin \cite{Mu} and \cite{Mu1} ,
the two-sided inequality
 \begin{eqnarray}\label{fih} 1-2\pi^2 H(X ,\frac{t }{2\pi})  \le |\p_X(t)|\le 1-4  H(X ,\frac{t }{2\pi})   ,
\end{eqnarray}
is established.  The following   is  the one-dimensional version of Theorem 5 in \cite{Mu}, which is stated without proof.
\begin{theorem} Let $X_1,\ldots, X_n$ have zero mean and finite third moments. Let
$$   H_n= \inf_{1/4\le d\le 1/2}\sum_{j=1}^n H(X_j
,d), \qq L_n= \frac{\sum_{j=1}^n\E |X_j|^3}{\big(\sum_{j=1}^n\E |X_j|^2\big)^{3/2}} .$$ Then
 $\D_n\le CL_n\, \big( {\Sigma_n }/{ H_n}\big)$. \end{theorem}
    The author further announced a manuscript
devoted to the question of the estimates of the rate of convergence. We have been  however unable to find   any corresponding
publication.
    For the iid case with third moment condition,  we
also  record Lemma 3 in Doney
\cite{D}.
\vskip 3 pt
Before stating our main result, say a few words concerning the method we will  use, which  is quite elementary. 
\vskip 3 pt
  Recall that $S_n=X_1+\ldots +X_n$,  where
$X_j$ are independent random variables such that
 $\P\{X_j
\in\mathcal L(v_{ 0},D )\}=1$, for the moment we do not assume any moment condition. We only suppose that
\begin{equation}\label{basber}  \t_{X_j}>0, \qq \quad  j=1,\ldots, n.
\end{equation}
  Anticipating a bit Lemma \ref{dec}, we can write
 $ S_n\buildrel{\mathcal D}\over{=}  W_n  +  DM_n
 $
where \begin{equation}\label{dec0}    W_n =\sum_{j=1}^n V_j,\qq M_n=\sum_{j=1}^n  \e_jL_j,  \quad
B_n=\sum_{j=1}^n
 \e_j .
\end{equation}The random variables
$ (V_j,\e_j),L_j$, $j=1,\ldots,n
 $    are mutually independent and $\e_j$, $ L_j $ are
independent Bernoulli random variables with $\P\{L_j =0\}=\P\{L_j=1\}=1/2$.  As moreover $M_n\buildrel{\mathcal D}\over{
=}\sum_{j=1}^{B_n } L_j$,   the following result will be relevant.
\begin{lemma}{\rm (\cite{[P]},  Chapter 7, Theorem 13)}  \label{lltber}Let
$\mathcal B_n=\b_1+\ldots+\b_n$, $n=1,2,\ldots$  where
$
\b_i
$ are i.i.d.  Bernoulli r.v.'s ($\P\{\b_i=0\}=\P\{\b_i=1\}=1/2$).  There exists a numerical constant $C_0$ such that for all positive $n$
 \begin{eqnarray*}  \sup_{z}\, \Big|  \P\big\{\mathcal B_n=z\}
 -\sqrt{\frac{2}{\pi n}} e^{-{ (2z-n)^2\over
2 n}}\Big| \le  \frac{C_0}{n^{3/2}}  .
  \end{eqnarray*}
\end{lemma}
\begin{remark} In fact  a little more is true,  namely that we have $ o ( {1}/{n^{3/2}} )$. 
 And it is also possible to show the following estimate yielding a better error term in presence of a different comparison term: There
exists an absolute constant
$C$ such that 
 \begin{equation}\sup_{k}\Big| \P\{ \mathcal B_n=z \} -\sqrt{{ 2 \over     \pi n}} 
 \int_\R e^{-i{2z-n\over \sqrt n}v-    {  v^2 \over   2}-    {  v^4 \over   12} }\dd v \Big|\le
C  {
\log^{7/2} n 
\over n^{ 5/2}}. 
\end{equation} \end{remark}

 Let $\E_{\!L}$, $\P_{\!L}$ (resp. $\E_{(V,\e)}$, $\P_{(V,\e)}$) stand  for the integration
symbols   and probability symbols relatively to the
$\s$-algebra generated by the sequence  $\{L_j , j=1, \ldots, n\}$   (resp. $\{(V_j,\e_j), j=1, \ldots, n\}$).
\vskip 2 pt 
Assume from now that the  $X_j$'s are   square integrable. The estimation of
$$\P \{S_n =\kappa \}   =     \E_{(V,\e)}     \P_{\!L}
\big\{D M_n  =\kappa-W_n
\big\} $$
relies upon  the conditional sum
 $S_n'= \E_L S_n= W_n + \frac{D}{2} B_n $, which verifies  
 \begin{eqnarray*} \E S'_n= \E S_n,  \qq \E  (S'_n) ^2=
 \E  S_n ^2- \frac{D^2\Theta_n }{4} .
\end{eqnarray*}

Set
\begin{eqnarray*} H_n&=& \sup_{x\in \R} \big|\P_{(V,\e)}
\big\{{S'_n-\E_{(V,\e)}S'_n  \over \sqrt{{\rm Var}(S'_n)}  }<x\big\} - \P\{g<x\} \big|\cr
  \rho_n(h)&=& \P\Big\{\big|\sum_{j=1}^n \e_j-\Theta_n\big|>h\Theta_n\Big\} ,
\end{eqnarray*}
where $\e_1,\ldots,\e_n$ are independent   random variables verifying $\P\{\e_j=1\}= 1-\P\{\e_j=0\}=\t_j$, $0<\t_j\le
\t_{X_j}$, $j=1,\ldots, n$ and
$\Theta_n=\sum_{j=1}^n \t_j$.
 As   $S_n'= \E_L S_n$, suitable moment  conditions permit to easily estimate $H_n$ by using Berry-Esseen estimates. And concentration
inequalities (Lemma \ref{primo}) provide sharp estimates of
$\rho_n(h)$.
\vskip 3 pt 
We are now ready to state our main result.   Let  $ C_1= \max (4,C_0  )$.
\begin{theorem}
\label{ger1}
   For any  $0<h<1$, $0<\t_j\le \t_{X_j}$,
 and all   $\k\in \mathcal L( v_{
0}n,D )$
 \begin{eqnarray*}      \P\{S_n =\k\} &\le  &
      \Big( \frac{1+ h  }{ 1-h}\Big) \, {  D \over
\sqrt{2  \pi  {\rm Var}(S_n)  } }  e^{-\frac{(\k- \E S_n)^2}{  2(1  +  h){\rm Var}(S_n)   } }
 \cr & &\quad   + {C_1
 \over
\sqrt{   (1-h)\Theta_n} } \big(H_n +      \frac{1}{(1-h)\Theta_n}
 \big) + \rho_n(h)     .
\end{eqnarray*}
And
\begin{eqnarray*}  \qq  \quad \P \{S_n =\kappa \}   &\ge &   \Big(\frac{   {      1- h   }}{       {  1 +h      }}\Big)  { D \over
\sqrt{2\pi {\rm Var}(S_n)     }}
 {  e^{- \frac{(\k-
 \E S_n)^2}{
 2(1-h){\rm Var}(S_n)    } }}
   - \cr & &\
  {C_1
 \over
\sqrt{   (1-h)\Theta_n} } \big(H_n  +   \frac{1}{(1-h)\Theta_n}  +  2\rho_n(h)
 \big) - \rho_n(h).
  \end{eqnarray*}
 \end{theorem}

 \begin{corollary}  \label{ger2}  Assume that $\frac{  \log \Theta_n }{\Theta_n}\le  {1}/{14} $. Then, for all $\k\in \mathcal L( v_{
0}n,D )$ such that
$$\frac{(\k- \E
S_n)^2}{    {\rm Var}(S_n)  } \le   \big({\frac { \Theta_n} {14 \log \Theta_n} }\big)^{1/2}   ,$$  we have
\begin{eqnarray*}  \Big| \P \{S_n =\kappa \} -{ D    e^{- \frac{(\k- \E
S_n)^2}{    2 {\rm Var}(S_n)    } }  \over \sqrt{2\pi {\rm Var}(S_n)     }} \Big|  & \le &   C_2\Big\{ D\big({    {   \log
\Theta_n }     \over
 {    {\rm Var}(S_n)   \Theta_n} } \big)^{1/2}  +    {    H_n +  \Theta_n^{-1}
\over \sqrt{   \Theta_n} } \Big\} .
   \end{eqnarray*}
And $C_2=12(C_1+1)$.
  \end{corollary}
\begin{remark}\label{123} Assume that
$$\lim_{n\to\infty}  \big(\frac{ {\rm Var}(S_n)     }{   \Theta_n  }\big)^{1/2}  \big(    H_n +  \frac{ 1     }{
\Theta_n  }\big)  =0.
   $$
This is for instance satisfied if
$${\rm (i)}\ \lim_{n\to\infty}   {\rm Var}(S_n)  =\infty, \qq {\rm (ii)}\ \lim_{n\to\infty}       H_n =0, \qq  {\rm (iii)}\
\limsup_{n\to\infty}
\frac{ {\rm Var}(S_n)     }{
\Theta_n  }  <\infty,
    $$  Then \begin{eqnarray*} \lim_{n\to\infty} \sup_{\frac{(\k- \E
S_n)^2}{    {\rm Var}(S_n)  } \le  ({\frac { \Theta_n} {14 \log \Theta_n} } )^{1/2}}\Big| \sqrt{ {\rm Var}(S_n)     }\P \{S_n
=\kappa \} -{ D    e^{- \frac{(\k- \E S_n)^2}{    2 {\rm Var}(S_n)    } }  \over \sqrt{2\pi     }} \Big|    & = & 0.
   \end{eqnarray*}
Now if $\frac{(\k- \E
S_n)^2}{    {\rm Var}(S_n)  } >  ({\frac { \Theta_n} {14 \log \Theta_n} } )^{1/2}$, then
$$e^{- \frac{(\k- \E S_n)^2}{    2 {\rm Var}(S_n)    } }\ll \frac{    2 {\rm Var}(S_n)    }{(\k- \E S_n)^2} \le  ({\frac  {14
\log \Theta_n} { \Theta_n} } )^{1/2}.$$
And
$$\P \{S_n =\kappa \}\le \P \big\{\frac{|S_n- \E S_n|}{      \sqrt{{\rm Var}(S_n) }   } \ge \frac{|\k- \E S_n|}{      \sqrt{{\rm
Var}(S_n) }   } \big\}\le   ({\frac  {14
\log \Theta_n} { \Theta_n} } )^{1/2}.$$
 Hence
$$\sqrt{ {\rm Var}(S_n)     }\P \{S_n
=\kappa \}\le ({\frac  {14
{\rm Var}(S_n)\log \Theta_n} { \Theta_n} } )^{1/2} $$
We deduce that if
$$\lim_{n\to\infty}  {\frac  {
{\rm Var}(S_n)\log \Theta_n} { \Theta_n} } =0, $$
then
\begin{eqnarray*} \lim_{n\to\infty} \sup_{ \k\in \mathcal L( v_{
0}n,D ) }\Big| \sqrt{ {\rm Var}(S_n)     }\P \{S_n
=\kappa \} -{ D    e^{- \frac{(\k- \E S_n)^2}{    2 {\rm Var}(S_n)    } }  \over \sqrt{2\pi     }} \Big|    & = & 0.
   \end{eqnarray*}
\end{remark}
\vskip 6pt Let $\psi:\R\to \R^+$ be even, convex and such that   $\frac
{\psi(x)}{x^2}$  and $\frac{x^3}{\psi(x)}$  are non-decreasing on $\R^+$ and assume now
  \begin{equation}\label{did}  \E \psi( X_j )<\infty    .
 \end{equation} Put $$L_n=\frac{  \sum_{j=1}^n\E \psi (X_j)  }
{   \psi (\sqrt
{ {\rm Var}(S_n )})}  .$$
Then Corollary \ref{ger2} can strengthened as follows
  \begin{corollary}  \label{ger3} Assume that $\frac{  \log \Theta_n }{\Theta_n}\le  {1}/{14} $. Then, for all $\k\in \mathcal L( v_{
0}n,D )$ such that
$$\frac{(\k- \E
S_n)^2}{    {\rm Var}(S_n)  } \le \sqrt{\frac{7 \log \Theta_n} {2\Theta_n}},$$  we have
\begin{eqnarray*}   \Big| \P \{S_n =\kappa \} -{ D e^{- \frac{(\k- \E
S_n)^2}{    2 {\rm Var}(S_n)    } } \over \sqrt{2\pi {\rm Var}(S_n)     }}     \Big|         & \le &    C_3\Big\{
D\big({    {   \log \Theta_n }     \over
 {    {\rm Var}(S_n)   \Theta_n} } \big)^{1/2}  +    {    L_n
 +  \Theta_n^{-1}
\over \sqrt{   \Theta_n} } \Big\} .
   \end{eqnarray*}
And $C_3=\max (C_2, 2^{  3/2}C_{{\rm E}}) $,   $C_{{\rm E}}$ being an absolute constant arising from  Esseen's inequality.
\end{corollary}
  \section{Preliminaries.}
\subsection{ \gsec Characteristics of a random variable.}       Let $X$ be a random variable such that
$\P\big(X
\in {\mathcal L}(v_0,D)\big)=1$  and recall according to    (\ref{vartheta}) that
$$ \t_X =\sum_{k\in \Z}\P\{X=v_k\}\wedge\P\{X=v_{k+1}\} .
$$
Then
\begin{eqnarray}\label{vartheta1}
0\le \t_X<1 .\end{eqnarray}
   Let indeed   $k_0$ be some integer such that
$f(k_0)>0$. Then
$$\sum_{k=k_0}^\infty  f(k)\wedge f(k+1)\le \sum_{k=k_0}^\infty  f(k+1)=\sum_{k=k_0+1}^\infty  f(k )
$$
And so  $0\le  \t_X\le       \sum_{k< k_0 }  f(k) +\sum_{k=k_0+1}^\infty  f(k )<1$.
\vskip 2 pt Now assume that $X$ has finite mean $\mu$ and finite variance
$\sigma^2$. The below inequality linking parameters $\s,D,\t_X$,  is implicit in
our proof (see   (\ref{prime})),
 \begin{eqnarray}\label{ssprime1}
\s^2\ge  \frac{ D^2   }{4} \t_X.
\end{eqnarray}
We begin with giving a proof valid  for
general lattice valued random variables. At
first by
 Tchebycheff's inequality,
$${{D^2}\over{4}}\P\big\{|X - \mu|\ge  {{D}\over{2}}\big\}\le  \sigma^2.$$
Now
\begin{eqnarray*} & & \P\big\{|X - \mu|\ge  \frac{D} {2} \big\}=\sum_{v_k\ge\mu+ {{D}\over{2}}
}\P\{X=v_k\}+ \sum_{v_k\le \mu- {{D}\over{2}} }\P\{X=v_k\}
\cr &\ge&
\!\!\sum_{v_{k+1}\ge\mu+ {{D}\over{2}} }\P\{X=v_{k})\wedge
\P\{X=v_{k+1}\}  + \!\!\!  \sum_{v_k\le \mu- {{D}\over{2}} }\P\{X=v_k\}\wedge
\P\{X=v_{k+1}\} \cr &=& \sum_{v_{k}\ge\mu- {{D}\over{2}}
}\P\{X=v_{k}\}\wedge \P\{X=v_{k+1}\}  +\!\!\!  \sum_{v_k\le \mu-
{{D}\over{2}} }\P\{X=v_k\}\wedge \P\{X=v_{k+1}\}
\cr&\ge &
 \t_X.
\end{eqnarray*}
Hence    inequality (\ref{ssprime1}).
\begin{remark} In Lemma 2 of  Mukhin \cite{Mu},    the following
inequality is proved
$$ \mathcal D(X,d):=\inf_{a\in \R}\E \langle (X-a)d\rangle^2 \ge  \frac{ |d|^2   }{4} \t_X, $$
where $d$ is a real number, $|d|\le 1/2$ and $\langle \a \rangle$ is the distance from $\a$ to the nearest integer.
Notice that $ \mathcal D(X,d)=0$ if and only if $X$ is lattice with span $1/d$.
 Let $\p_X(t)=
 \E e^{i t X}$.  \end{remark}

\subsection{ \gsec Bernoulli component of a random variable}  \label{bercomp}
  Let $X$ be
  a random variable   such that  $\P\{X
\in\mathcal L(v_0,D)\}=1$. It is not necessary to suppose here that the  span $D$ is maximal.     Put
$$ f(k)= \P\{X= v_k\}, \qq k\in \Z .$$
We assume that   \begin{equation}\label{basber1}\t_X>0.
  \end{equation}
 Notice that $
\t_X<1$.  Indeed, let   $k_0$ be some integer such that
$f(k_0)>0$. Then
$$\sum_{k=k_0}^\infty  f(k)\wedge f(k+1)\le \sum_{k=k_0}^\infty  f(k+1)=\sum_{k=k_0+1}^\infty  f(k )
$$
And so  $ \t_X\le       \sum_{k< k_0 }  f(k) +\sum_{k=k_0+1}^\infty  f(k )<1$. Let $0<\t\le\t_X$. One can associate to $\t$ and $X$  a
sequence $  \{ \tau_k, k\in \Z\}$     of   non-negative reals such that
\begin{equation}\label{basber0}  \tau_{k-1}+\tau_k\le 2f(k), \qq  \qq\sum_{k\in \Z}  \tau_k =\t.
\end{equation}
Just take
 $\tau_k=  \frac{\t}{\nu_X} \, (f(k)\wedge f(k+1))  $. 
   Now   define   a pair of random variables $(V,\e)$   as follows:
  \begin{eqnarray}\label{ve} \qq\qq\begin{cases} \P\{ (V,\e)=( v_k,1)\}=\tau_k,      \cr
 \P\{ (V,\e)=( v_k,0)\}=f(k) -{\tau_{k-1}+\tau_k\over
2}    .  \end{cases}\qq (\forall k\in \Z)
\end{eqnarray}
   By assumption     this is   well-defined,   and the margin  laws verify
\begin{eqnarray}\begin{cases}\P\{ V=v_k\} &= \  f(k)+ {\tau_{k }-\tau_{k-1}\over 2} ,
\cr
 \P\{ \e=1\} &= \ \t \ =\ 1-\P\{ \e=0\}   .
\end{cases}\end{eqnarray}
Indeed, $\P\{ V=v_k\}= \P\{ (V,\e)=( v_k,1)\}+ \P\{ (V,\e)=( v_k,0)\}=f(k)+ {\tau_{k }-\tau_{k-1}\over 2} .$
 Further  $\P\{ \e=1\}  =\sum_{k\in\Z} \P\{ (V,\e)=( v_k,1)\}=\sum_{k\in\Z} \tau_{k }=\t  $.
 \vskip 3pt  The whole approach is based on the lemma below, the proof of which is given for the sake of completeness.
\begin{lemma} \label{bpr} Let $L$
be a Bernoulli random variable    which is independent from $(V,\e)$, and put
 $Z= V+ \e DL$.
We have $Z\buildrel{\mathcal D}\over{ =}X$.
\end{lemma}
\begin{proof} (\cite{MD},\cite{W}) Plainly,
\begin{eqnarray*}\P\{Z=v_k\}&=&\P\big\{ V+\e DL=v_k, \e=1\}+ \P\big\{ V+\e DL=v_k, \e=0\} \cr
&=&{\P\{ V=v_{k-1}, \e=1\}+\P\{
V=v_k, \e=1\}\over 2} +\P\{ V=v_k, \e=0\}
\cr&=& {\tau_{k-1}+ \tau_{k }\over 2} +f(k)-{\tau_{k-1}+ \tau_{k
}\over 2}
= f(k).
\end{eqnarray*}
\end{proof}
 }

Now consider independent random variables  $ X_j,j=1,\ldots,n$,   each   satisfying   assumption  (\ref{basber1})
 and let $0<\t_i\le \t_{X_i}$, $i=1,\ldots, n$. Iterated  applications of Lemma \ref{bpr} allow to
 associate to them a
sequence of independent vectors $ (V_j,\e_j, L_j) $,   $j=1,\ldots,n$  such that
 \begin{eqnarray}\label{dec0} \big\{V_j+\e_jD   L_j,j=1,\ldots,n\big\}&\buildrel{\mathcal D}\over{ =}&\big\{X_j, j=1,\ldots,n\big\}  .
\end{eqnarray}
Further the sequences $\{(V_j,\e_j),j=1,\ldots,n\}
 $   and $\{L_j, j=1,\ldots,n\}$ are independent.
For each $j=1,\ldots,n$, the law of $(V_j,\e_j)$ is defined according to (\ref{ve}) with $\t=\t_j$.  And $\{L_j, j=1,\ldots,n\}$ is  a sequence  of
independent Bernoulli random variables. Set
\begin{equation}\label{dec} S_n =\sum_{j=1}^n X_j, \qq  W_n =\sum_{j=1}^n V_j,\qq M_n=\sum_{j=1}^n  \e_jL_j,  \quad B_n=\sum_{j=1}^n
 \e_j .
\end{equation}
  \begin{lemma} \label{lemd}We have the
representation
\begin{eqnarray*} \{S_k, 1\le k\le n\}&\buildrel{\mathcal D}\over{ =}&  \{ W_k  +  DM_k, 1\le k\le n\} .
\end{eqnarray*}
And  $M_n\buildrel{\mathcal D}\over{ =}\sum_{j=1}^{B_n } L_j$.
 \end{lemma}

  \section{Proof of Theorem \ref{ger1}}
 We   denote again $X_j= V_j+D\e_jL_j$, $S_n= W_n  +  M_n$,
$j,\, n\ge 1$, which is justified by the previous representation.
      Fix $0<h<1$ and let
  \begin{eqnarray}\label{dep0}A_n=\Big\{\big| B_n - \Theta_n\big|\le h\Theta_n
\Big\}, \qq\qq  \rho_n(h)= \P_{(V,\e)}(A_n^c)   .
\end{eqnarray}
For $\k\in \mathcal L(v_0,D)$,
  \begin{eqnarray}\label{dep} \P \{S_n =\kappa \}   &=& \E_{(V,\e)}      \P_{\!L}
\big\{D \sum_{j= 1}^n \e_jL_j  =\kappa-W_n
\big\}
\cr &=&\E_{(V,\e)}   \Big( \chi(A_n)+\chi(A_n^c)\Big)    \P_{\!L}
\big\{D \sum_{j= 1}^n \e_jL_j  =\kappa-W_n
\big\} . \end{eqnarray}
 Thus
 \begin{eqnarray}\label{dep0} \Big|\P \{S_n =\kappa \}  -  \E_{(V,\e)}   \chi(A_n)    \P_{\!L}
\big\{D \sum_{j= 1}^n \e_jL_j  =\kappa-W_n
\big\}\Big|&\le &  \P_{(V,\e)}(A_n^c)\cr &= & \rho_n(h)
. \end{eqnarray}
  We have  $\sum_{j= 1}^n \e_jL_j\buildrel{\mathcal D}\over{ =}\sum_{j=1}^{B_n } L_j$.
 In view of Lemma \ref{lltber}, \begin{eqnarray*}  \sup_{z}\, \Big|  \P_{\!L}\big\{\sum_{j=1}^{N } L_j=z  \big\} -{2\over \sqrt{2\pi
N}}e^{-{( z-(N/2))^2\over
 N/2}}\Big|
 \le {C_0\over N^{3/2}} .
 \end{eqnarray*}
    On $A_n$,   we    have   $   (1-h )\Theta_n \le  B_n   \le     (1+h )\Theta_n$. 
 Therefore
\begin{eqnarray} \label{dep2} \Big|\E_{(V,\e)}     \chi(A_n) \Big\{     \P_{\!L}
\big\{D \sum_{j= 1}^n \e_jL_j  =\kappa-W_n
\big\} - {2e^{-{(\kappa-W_n-D(B_n/2))^2\over
D^2(B_n/2)}}\over
\sqrt{2\pi B_n}} \Big\}\Big|\cr  \le   C_0\ \E_{(V,\e)}  \chi(A_n)\cdot  B_n^{-3/2}\le     \frac{ C_0}{    (1-h )^{ 3/2}} \,\frac{
1}{(\sum_{i=1}^n \t_i )^{ 3/2}
 }
  .
\end{eqnarray}
And by inserting this into (\ref{dep0})
 \begin{eqnarray}\label{dep01}  \Big|\P \{S_n =\kappa \}  -  \E_{(V,\e)}   \chi(A_n)   {2e^{-{(\kappa-W_n-D(B_n/2))^2\over
D^2(B_n/2)}}\over
\sqrt{2\pi B_n}}  \Big|  &\le & \frac{ C_0}{     (1-h )^{ 3/2}} \,\frac{ 1}{(\sum_{i=1}^n \t_i )^{ 3/2}
 }
 + \rho_n(h)
.
  \end{eqnarray}
\vskip 5 pt \noi {\it Step 2.} ({\it Second reduction})
  Some elementary algebra is necessary in order to put the exponential term in a more appropriate form. Recall that $S_n= W_n  +
M_n$.
\begin{lemma}\label{mmprime}
  Let
$S'_n=W_n+D( B_n/2)$.  Then
\begin{eqnarray*} \E S'_n= \E S_n,  \qq \E  (S'_n) ^2=
 \E  S_n ^2- \frac{D^2\Theta_n }{4} .
\end{eqnarray*}
\end{lemma}
 \begin{proof} At first
$\E S_n=\E_{(V,\e)}\E_{\!L}\big(W_n+  D\sum_{j=1}^{ n}  \e_jL_j\big) =
\E_{(V,\e)}\big(W_n+ D{ B_n\over 2}\big) =\E S'_n
$. Further, by using independence,
  \begin{eqnarray*}\E_{(V,\e)} B_n^2&=&    \sum_{1\le i,j\le n\atop i\not = j}  \E_{(V,\e)} \e_i\E_{(V,\e)}\e_j +\sum_{1\le i \le n }   \E_{(V,\e)}\e_i^2
\cr &=&    \sum_{1\le i,j\le n\atop i\not = j}
\t_{ i}\t_{j} +\sum_{i=1}^n   \t_{ i}
=\Big(\sum_{i=1}^n  \t_{ i}\Big)^2-\sum_{i=1}^n  \t_{ i}^2 +\sum_{i=1}^n   \t_{ i}  ,
\end{eqnarray*} and
 \begin{eqnarray*}\E \Big(\sum_{j=1}^n \e_j L_j\Big)^2&=& \sum_{1\le i,j\le n\atop i\not = j} \E_{(V,\e)}\e_i\e_j \E_L
L_i\E_L L_j +\sum_{i=1}^n  \E_{(V,\e)}\e_i^2 \E_L
L_i^2
\cr
&=& \frac{1}{4}\sum_{1\le i,j\le n\atop i\not = j}  \E_{(V,\e)}\e_i\e_j   +\frac{1}{2}\sum_{i=1}^n  \E_{(V,\e)}\e_i^2
= \frac{1}{4}\sum_{1\le i,j\le n\atop i\not = j}  \t_{ i}\t_{j}   +\frac{1}{2}\sum_{i=1}^n \t_{ i} \cr
\cr
&=& \frac{1}{4}\Big\{\Big(\sum_{i=1}^n   \t_{ i}\Big)^2-\sum_{i=1}^n  \t_{ i}^2\Big\}    +\frac{1}{2}\sum_{i=1}^n
\t_{ i} .
   \end{eqnarray*}

Now
\begin{eqnarray*}\E  S_n ^2 &=&\E \big( W_n+  D \sum_{i=1}^n  \e_i L_i\big) ^2
\cr &=&\E_{(V,\e)}  W_n^2 +2D \E_{(V,\e)}W_n\Big( \sum_{i=1}^n    \e_i \E_L   L_i\Big)+  D^2 \E \Big(\sum_{i=1}^n  \e_i L_i\Big)^2
 \cr &=& \E_{(V,\e)} \Big(W_n^2+2D W_n \big({ B_n\over 2} \big) \Big)+
\frac{D^2}{4}\Big\{\Big(\sum_{i=1}^n  \t_{ i}\Big)^2-\sum_{i=1}^n  \t_{ i}^2\Big\}
      +\frac{D^2}{2}\sum_{i=1}^n \t_{ i} .
\end{eqnarray*}
And
\begin{eqnarray*} \E  (S'_n) ^2&=& \E_{(V,\e)} \big(W_n^2+2DW_n \big({ B_n\over 2} \big)   \big)+ \frac{D^2}{4}  \Big\{\Big(\sum_{i=1}^n
 \t_{ i}\Big)^2-\sum_{i=1}^n  \t_{ i}^2
 +\sum_{i=1}^n   \t_{ i} \Big\}
\cr &=& \E  S_n ^2- \frac{D^2}{4}\sum_{i=1}^n \t_{ i}.
\end{eqnarray*}  Hence Lemma \ref{mmprime}  is established.
\end{proof}

  We   deduce
\begin{equation} \label{prime} {\rm Var}(S'_n)={\rm Var}(S_n)-
\frac{D^2}{4}\sum_{i=1}^n \t_{ i}= \sum_{i=1}^n\Big(\s_i^2-\frac{D ^2\t_{ i}}{4}\Big)
\end{equation}
 Put
 \begin{equation*} T_n= {S'_n-\E_{(V,\e)}S'_n  \over \sqrt{{\rm Var}(S'_n)}  }.
\end{equation*}
As $\E_{(V,\e)}S'_n=\E S_n$ we can   write
\begin{eqnarray*}  {(\kappa-W_n- D(B_n/2))^2\over  D^2(B_n/2)} &=& {(\kappa-\E S_n -\{S'_n-\E_{(V,\e)}S'_n\})^2\over  D^2(B_n/2)}
\cr & = &{ {\rm
Var}(S'_n) \over  D^2 (B_n/2) }\Big({\kappa-\E S_n \over \sqrt{{\rm Var}(S'_n)}  } - T_n \Big)^2
\end{eqnarray*}
 And (\ref{dep01}) is more conveniently rewritten as
 \begin{eqnarray}\label{dep21} \Big|\P \{S_n =\kappa \}  -  \Upsilon_n\Big|&\le &  \frac{ C_0}{     (1-h )^{ 3/2}} \,\frac{
1}{(\sum_{i=1}^n \t_i )^{ 3/2}
 }
 + \rho_n(h),
  \end{eqnarray}
where
\begin{eqnarray}\label{dep210}  \Upsilon_n=  \E_{(V,\e)}   \chi(A_n) {2e^{-{
 {\rm Var}(S'_n) \over  D^2 (B_n/2) }\big({\kappa-\E S_n \over \sqrt{{\rm Var}(S'_n)}  } - T_n\big)^2}\over
\sqrt{2\pi B_n}}
.
\cr & & \end{eqnarray}
  Set for $-1< u\le 1$,
 $$ Z_n(u)=   \E_{(V,\e)}       e^{-{2{\rm Var}(S'_n)\over  D^2  (1 + u ) \Theta_n   } \big({\kappa-\E S_n
\over
\sqrt{{\rm Var}(S'_n)}  }-  T_n\big)^2}.  $$
Then
\begin{eqnarray} \label{dep3}      { 2Z_n(-h) - 2\rho_n(h) \over
\sqrt{2\pi    (1 +h )\Theta_n}}\ \le \   \Upsilon_n  &\le &{ 2Z_n(h) \over
\sqrt{2\pi    (1 - h )  \Theta_n}} .
 \end{eqnarray}
The second inequality is obvious, and  the first follows from

 \begin{eqnarray*}
  \Upsilon_n  &\ge &   { 2 \over
\sqrt{2\pi     (1 +h ) \Theta_n}}\E_{(V,\e)}    \chi(A_n)     e^{-{2{\rm Var}(S'_n)\over   D^2  (1 -h )  \Theta_n  } \big({\kappa-\E S_n
\over \sqrt{{\rm Var}(S'_n)}  }-  T_n\big)^2}
 \cr   &\ge  &   { 2 \over
\sqrt{2  \pi(1 +h )    \Theta_n}}\bigg\{\E_{(V,\e)}       e^{-{2{\rm Var}(S'_n)\over   D^2  (1 -h )  \Theta_n } \big({\kappa-\E S_n
\over \sqrt{{\rm Var}(S'_n)}  }-  T_n\big)^2} -\P_{(V,\e)} (A_n^c)\bigg\}
  \cr &\ge  &      { 2  Z_n(-h) -2\rho_n(h) \over
\sqrt{2\pi    (1 +h )\Theta_n}}.
\end{eqnarray*}

\vskip 5 pt \noi {\it Step 3.}
  ({\it Exponential moment})
 \begin{lemma}\label{tech} Let $Y  $ be a centered
  random variable. For any positive reals
$a$ and
$b$
\begin{eqnarray*} \Big|\E  e^{-a(b-Y)^2} - \frac{e^{-  \frac{b^2}{2+ 1/a} }}{ \sqrt{1+2a}}\Big|&\le &  4\sup_{x\in \R} \big|\P
 \{Y<x \} - \P\{g<x\} \big| .
\end{eqnarray*}\end{lemma}
\begin{proof}
By   the transfert formula,
\begin{eqnarray*}  \Big|\E  e^{-a(b-Y)^2}-\E  e^{-a(b-g)^2}\Big|  
& =&\Big|\int_0^1  \Big( \P \big\{e^{-a(b-Y)^2}>x\big\}- \P \big\{e^{-a(b-g)^2}>x\big\} \Big)\dd x\Big|
\cr \  (x=e^{-ay^2})\quad &=&2a \Big|\int_0^\infty  \Big( \P \big\{|b-Y|<y\big\}-\P \big\{|b-g|<y\big\}\Big)
ye^{-ay^2} \dd y\Big|
     \cr &\le   &     4\sup_{x\in \R} \Big|\P
\big\{Y<x\big\} - \P\{g<x\} \Big|
   .
\end{eqnarray*}
 The claimed estimate follows from\begin{equation}\label{estexp}\E e^{-a(b-g)^2}= \frac{e^{-  \frac{b^2}{2+ 1/a} }}{ \sqrt{1+2a}} .\end{equation}
 \end{proof}
We apply Lemma \ref{tech} to estimate $Z_n(u)$. Here
 $ a= {2{\rm
Var}(S'_n)\over D^2 (1+u ) \Theta_n}$, $b=
\frac{\k- \E S_n}{\sqrt{{\rm Var}(S'_n)}}  $. 
  Since   by (\ref{prime}), ${\rm
Var}(S'_n)= {\rm Var}(S_n)- \frac{ D^2 \Theta_n}{4}$,   we have
\begin{eqnarray*} \frac{b^2}{2+ 1/a} &= &\frac{(\k- \E S_n)^2}{ {\rm Var}(S'_n)\big( 2+  \frac{D^2 (1+u )\Theta_n }{2{\rm
Var}(S'_n)}\big)} = \frac{(\k- \E S_n)^2}{    2{\rm Var}(S'_n)  +  \frac{D^2(1+u )\Theta_n }{2 } }
\cr
& = & \frac{(\k-
\E S_n)^2}{
 2{\rm Var}(S_n)- \frac{ D^2 \Theta_n}{2}  +  \frac{D^2    (1+u )\Theta_n }{2 } }=\frac{(\k-
\E S_n)^2}{
2{\rm Var}(S_n) +  \frac{D^2  u  \Theta_n }{2 } }
\cr
& = &
\frac{(\k- \E S_n)^2}{   2{\rm Var}(S_n)(1  +  \d(u)) },  \end{eqnarray*}
where we have denoted
$$\d(u)= \frac{D^2      \Theta_n  u}{4 {\rm Var}(S_n) } .$$
 Now
 \begin{eqnarray*}\frac{1}{ \sqrt{1+2a}}& =& \Big(\frac{1}{  {1+ {4{\rm Var}(S'_n)\over D^2
  (1+u ) \Theta_n}}}\Big)^{1/2} =\frac{D}{ 2    }\Big(\frac{    { (1+ u  )\Theta_n}}{    {    {\rm Var}(S_n') +{D^2
  (1+ u  ) \Theta_n\over 4} }}\Big)^{1/2}\cr
&=&\frac{D}{ 2    }\Big(\frac{   { (1+ u  )\Theta_n}}{    {   {\rm Var}(S_n) +{D^2
   h \Theta_n \over 4} }}\Big)^{1/2}
 =\frac{D}{ 2    }\Big(\frac{   { \Theta_n   (1+ u  )}}{       {{\rm Var}(S_n)  (1+ \d(u)  ) }}\Big)^{1/2} .
 \end{eqnarray*}

This along with Lemma \ref{tech}   provides the following bound,
   \begin{eqnarray} \label{fb}    \Big| Z_n(u) -   \frac{D}{ 2    }\Big(\frac{   { \Theta_n   (1+ u  )}}{       {{\rm Var}(S_n)  (1+
\d(u)  ) }}\Big)^{1/2} e^{- \frac{(\k- \E S_n)^2}{   2{\rm Var}(S_n)(1  +  \d(u)) } }  \Big| &\le  &    4H_n , \end{eqnarray}
 with
$$ H_n=\sup_{x\in \R} \Big|\P_{(V,\e)}
\big\{T_n<x\big\} - \P\{g<x\} \Big| .$$
Besides, it follows from (\ref{ssprime1}) that for $h\ge 0$,
 \begin{eqnarray} \label{fb1} 0\le \d(h) \le h .
\end{eqnarray}
\vskip 5 pt \noi {\it Step 4.}({\it Conclusion}) Consider the upper bound part.
 By reporting (\ref{fb}) into (\ref{dep3})  and using  (\ref{fb1}), we get
   \begin{eqnarray*}   \Upsilon_n  &\le &
  { 8H_n   \over \sqrt{2  \pi  (1-h)\Theta_n} }      +   \Big( \frac{1+ h  }{ 1-h}\Big) \, {  D \over
\sqrt{2  \pi  {\rm Var}(S_n)  } }  e^{-\frac{(\k- \E S_n)^2}{  2(1  +  h){\rm Var}(S_n)   } }
  .
\end{eqnarray*}
  And by combining   with (\ref{dep21}),
  \begin{eqnarray} \label{dep5}    \P\{S_n =\k\} &\le  &
      \Big( \frac{1+ h  }{ 1-h}\Big) \, {  D \over
\sqrt{2  \pi  {\rm Var}(S_n)  } }  e^{-\frac{(\k- \E S_n)^2}{  2(1  +  h){\rm Var}(S_n)   } }
 \cr & &\quad   + { 8H_n
 \over
\sqrt{2  \pi  (1-h)\Theta_n} } +   \frac{ C_0}{     (1-h )^{ 3/2}} \,\frac{ 1}{ \Theta_n  ^{ 3/2}
 }
 + \rho_n(h)     .
\end{eqnarray}
     Similarly,  by using
(\ref{dep3}),
\begin{eqnarray} \label{dep6}
 \Upsilon_n
  &\ge  &
\Big(\frac{   {      1- h   }}{       {  1 +h      }}\Big)  { D \over
\sqrt{2\pi {\rm Var}(S_n)     }}
 {  e^{- \frac{(\k-
 \E S_n)^2}{
 2(1-h){\rm Var}(S_n)    } }}
   - {8  H_n+ 2\rho_n(h) \over
\sqrt{2\pi    (1 +h )\Theta_n}}  .
 \end{eqnarray}

By combining  with  (\ref{dep21}), we obtain
\begin{eqnarray}\label{dep5}  \P \{S_n =\kappa \}   &\ge &   \Big(\frac{   {      1- h   }}{       {  1 +h      }}\Big)  { D \over
\sqrt{2\pi {\rm Var}(S_n)     }}
 {  e^{- \frac{(\k-
 \E S_n)^2}{
 2(1-h){\rm Var}(S_n)    } }}
   - { 8  H_n+ 2\rho_n(h) \over
\sqrt{2\pi    (1 +h )\Theta_n}}\cr & &\ -\frac{ C_0}{     (1-h )^{ 3/2}} \,\frac{
1}{\Theta_n^{ 3/2}
 }
 - \rho_n(h),
  \end{eqnarray}
 As  $  \max ({ 8
 /\sqrt{2  \pi  } },C_0  )\le C_1 $, we deduce
\begin{eqnarray} \label{f1}    \P\{S_n =\k\} &\le  &
      \Big( \frac{1+ h  }{ 1-h}\Big) \, {  D \over
\sqrt{2  \pi  {\rm Var}(S_n)  } }  e^{-\frac{(\k- \E S_n)^2}{  2(1  +  h){\rm Var}(S_n)   } }
 \cr & &\quad   + {C_1
 \over
\sqrt{   (1-h)\Theta_n} } \big(H_n +      \frac{1}{(1-h)\Theta_n}
 \big) + \rho_n(h)     .
\end{eqnarray}
And
\begin{eqnarray}\label{f2}  \P \{S_n =\kappa \}   &\ge &   \Big(\frac{   {      1- h   }}{       {  1 +h      }}\Big)  { D \over
\sqrt{2\pi {\rm Var}(S_n)     }}
 {  e^{- \frac{(\k-
 \E S_n)^2}{
 2(1-h){\rm Var}(S_n)    } }}
   - \cr & &\
  {C_1
 \over
\sqrt{   (1-h)\Theta_n} } \big(H_n  +   \frac{1}{(1-h)\Theta_n}  +  2\rho_n(h)
 \big) - \rho_n(h).
  \end{eqnarray}
This achieves the proof.
\section{\bf Proof of Corollary \ref{ger2}} In order to estimate $\rho_n(h)$ we use the following Lemma  (\cite{di}, Theorem 2.3)
\begin{lemma} \label{primo}
Let $X_1, \dots, X_n$     be independent random variables, with $0 \le X_k \le 1$ for each $k$.
Let $S_n = \sum_{k=1}^n X_k$ and $\mu = E[S_n]$. Then for any $\epsilon >0$,
 \begin{eqnarray*}
(a) &&
 \P\big\{S_n \ge  (1+\epsilon)\mu\big\}
    \le  e^{- \frac{\epsilon^2\mu}{2(1+ \epsilon/3) } } .
\cr (b) & &\P\big\{S_n \le  (1-\epsilon)\mu\big\}\le    e^{- \frac{\epsilon^2\mu}{2}}.
 \end{eqnarray*}
\end{lemma}

By (a) and (b), and observing that
 $ e^{- \frac{\epsilon^2\mu}{2}}\le  e^{- \frac{\epsilon^2\mu}{2(1+ \epsilon/3)}}$,
we obtain
\begin{eqnarray*}  \rho_n(h) &=& \P\big\{\big|\sum_{k=1}^n \e_k- \Theta_n\big|> h \Theta_n\big\}
=\P\big\{\sum_{k=1}^n \e_k>(1+ h) \Theta_n\big\}+ \P\big\{\sum_{k=1}^n \e_k<(1- h) \Theta_n\big\}
 \cr &\le  & 2 e^{- \frac{h^2\Theta_n}{2(1+ h/3)}}. \end{eqnarray*}
 Let  $ h_n=\sqrt{\frac{7 \log \Theta_n} {2\Theta_n}}$.
 By assumption  $\frac{  \log \Theta_n }{\Theta_n}\le  {1}/{14} $.   Thus
$h_n\le 1/2$ and so  $\frac{h_n^2\Theta_n}{2(1+ h_n/3)}\ge  {(3/2)\log
\Theta_n}
 $. It follows that
\begin{eqnarray*}  \rho_n(h)   &\le  & 2 \Theta_n^{- 3/2}. \end{eqnarray*}
   Further    \begin{eqnarray*}   {C_1
 \over
\sqrt{   (1-h)\Theta_n} } \big(H_n +      \frac{1}{(1-h)\Theta_n}
 \big) + \rho_n(h) &\le &   2^{  1/2}C_1{  H_n
 \over
\sqrt{   \Theta_n} }  +    {2^{  3/2}C_1 +2
 \over
 \Theta_n^{  3/2} }    .
\end{eqnarray*}
Let $C_2=2^{  3/2  } 3(C_1+1)
  $. Therefore
\begin{eqnarray*}      \P\{S_n =\k\} &\le  &         {  D  ( 1+4h_n )\over
\sqrt{2  \pi  {\rm Var}(S_n)  } }  e^{-\frac{(\k- \E S_n)^2}{  2(1  +  h_n){\rm Var}(S_n)   }}     +   {C_2  \over \sqrt{   \Theta_n}
}\big( H_n +    {1 \over  \Theta_n  }   \big)    .
\end{eqnarray*}
 Besides
\begin{eqnarray*}
& &  {C_1
 \over
\sqrt{   (1-h)\Theta_n} } \big(H_n  +   \frac{1}{(1-h)\Theta_n}  +  2\rho_n(h)
 \big) + \rho_n(h)
\cr &\le & 2^{1/2}C_1   { H_n
 \over
\sqrt{   \Theta_n} }    +    {2  (3.2^{1/2}C_1 +1)
 \over
 \Theta_n^{  3/2} }\le        {C_2
 \over
\sqrt{   \Theta_n} } \big(  H_n    +    { 1
 \over
 \Theta_n  }  \big).
  \end{eqnarray*}
Consequently,
\begin{eqnarray*}   \P \{S_n =\kappa \}   &\ge &      { D(1-2h_n) \over \sqrt{2\pi {\rm Var}(S_n)     }}      {  e^{- \frac{(\k-  \E
S_n)^2}{  2(1-h_n){\rm Var}(S_n)    } }}   -    { C_2   \over \sqrt{   \Theta_n} } \big(H_n +    {1 \over \Theta_n  }\big) .
  \end{eqnarray*}

If $$  \frac{(\k- \E S_n)^2}{  2 {\rm
Var}(S_n)  }
\le \frac{   1+h_n   }{ h_n} , $$
 then by using the   inequalities $e^u\le 1+3u$ and $Xe^{-X}\le e^{-1}$ valid for $0\le u\le 1$,   $X\ge 0$, we   get
 \begin{eqnarray*}  e^{- \frac{(\k- \E S_n)^2}{   2(1-h_n){\rm Var}(S_n)    } }
&= &    e^{- \frac{(\k- \E S_n)^2}{    2 {\rm Var}(S_n)    } } e^{\frac{(\k- \E S_n)^2}{  2 {\rm
Var}(S_n)  }\frac{ h_n}{   1+h_n   } }
\cr&\le &    e^{- \frac{(\k- \E S_n)^2}{    2 {\rm Var}(S_n)    } }   \Big\{ 1+ 3\frac{(\k- \E S_n)^2}{  2 {\rm
Var}(S_n)  }\frac{ h_n}{   1+h_n   }
\Big\}
 \cr&\le &      e^{- \frac{(\k- \E S_n)^2}{    2 {\rm Var}(S_n)    } }+ \frac{ 3h_n}{   e(1+h_n  ) }  \le  e^{- \frac{(\k- \E
S_n)^2}{    2 {\rm Var}(S_n)    } }+   2h_n .
\end{eqnarray*}
Hence,
\begin{eqnarray*}    {  D  ( 1+4h_n )\over
\sqrt{2  \pi  {\rm Var}(S_n)  } } e^{- \frac{(\k- \E S_n)^2}{   2(1-h_n){\rm Var}(S_n)    } }
 &\le &     {  D  ( 1+4h_n )\over
\sqrt{2  \pi  {\rm Var}(S_n)  } }   \big\{  e^{- \frac{(\k- \E S_n)^2}{    2 {\rm Var}(S_n)    } }+ 2h_n
\big\}
 \cr&\le &   {  D  e^{- \frac{(\k- \E S_n)^2}{    2 {\rm Var}(S_n)    } }   \over
\sqrt{2  \pi  {\rm Var}(S_n)  } }   +{  4h_n D   \over
\sqrt{2  \pi  {\rm Var}(S_n)  } }+ { 4h_n D  ( 1+2h_n )\over
\sqrt{2  \pi  {\rm Var}(S_n)  } }
 \cr&\le &   {  D  e^{- \frac{(\k- \E S_n)^2}{    2 {\rm Var}(S_n)    } }   \over
\sqrt{2  \pi  {\rm Var}(S_n)  } }   +{  16h_n D   \over
\sqrt{2  \pi  {\rm Var}(S_n)  } }
.\end{eqnarray*}
Therefore,
 recalling that $ h_n=\sqrt{\frac{7 \log \Theta_n} {2\Theta_n}}$,
\begin{eqnarray*}      \P\{S_n =\k\} -     {  D  e^{- \frac{(\k- \E S_n)^2}{    2 {\rm Var}(S_n)    } }   \over
\sqrt{2  \pi  {\rm Var}(S_n)  } }  &\le  & {  16h_n D   \over
\sqrt{2  \pi  {\rm Var}(S_n)  } }   +   C_2\, {     H_n +  \Theta_n^{-1}
\over \sqrt{   \Theta_n} }
 \cr   &\le   & C_3\Big\{ D\big({    {   \log \Theta_n }     \over
 {    {\rm Var}(S_n)   \Theta_n} } \big)^{1/2}  +    {     H_n +  \Theta_n^{-1}
\over \sqrt{   \Theta_n} } \Big\} .
\end{eqnarray*}
since $     8\sqrt{ 7     / \pi  }   \le C_2$.
 Similarly, if $$  \frac{(\k- \E S_n)^2}{  2 {\rm
Var}(S_n)  } \le \frac{   1    }{2 h_n} , $$
 {\rm
 then
\begin{eqnarray*}  e^{- \frac{(\k- \E S_n)^2}{   2(1-h_n){\rm
Var}(S_n)    } } &\ge  &    e^{- \frac{(\k- \E S_n)^2}{    2 {\rm Var}(S_n)    } } e^{-h_n\frac{(\k- \E S_n)^2}{    {\rm
 Var}(S_n)  }  }
 \ge      e^{- \frac{(\k- \E S_n)^2}{    2 {\rm Var}(S_n)    } }   \Big\{ 1- 3h_n\frac{(\k- \E S_n)^2}{    {\rm
 Var}(S_n)  }
\Big\}
 \cr&\ge &      e^{- \frac{(\k- \E S_n)^2}{    2 {\rm Var}(S_n)    } }- \frac{ 6h_n}{   e  }  \ge  e^{- \frac{(\k- \E
S_n)^2}{    2 {\rm Var}(S_n)    } }-   3h_n ,
\end{eqnarray*}
where we   used the inequality $\frac{1}{1+3u}\ge 1-3u $.

Hence,
  \begin{eqnarray*}   { D(1-2h_n) \over \sqrt{2\pi {\rm Var}(S_n)     }}      {  e^{- \frac{(\k-  \E
 S_n)^2}{  2(1-h_n){\rm Var}(S_n)    } }}
 &\ge &    { D(1-2h_n) \over \sqrt{2\pi {\rm Var}(S_n)     }}   \big\{ e^{- \frac{(\k- \E
S_n)^2}{    2 {\rm Var}(S_n)    } }-   3h_n
\big\}
 \cr &\ge &   { D \over \sqrt{2\pi {\rm Var}(S_n)     }}    e^{- \frac{(\k- \E
S_n)^2}{    2 {\rm Var}(S_n)    } } -{5h_n D  \over \sqrt{2\pi {\rm Var}(S_n)     }}
    .\end{eqnarray*}
Consequently,
  \begin{eqnarray*}   \P \{S_n =\kappa \} -{ D \over \sqrt{2\pi {\rm Var}(S_n)     }}    e^{- \frac{(\k- \E
S_n)^2}{    2 {\rm Var}(S_n)    } }   &\ge &     -{5h_n D  \over \sqrt{2\pi {\rm Var}(S_n)     }}   -    { C_2   \over \sqrt{
\Theta_n} } \big(H_n +    {1 \over \Theta_n  }\big)
\cr &\ge &-C_3\Big\{ D\big({    {   \log \Theta_n }     \over
 {    {\rm Var}(S_n)   \Theta_n} } \big)^{1/2}  +    {     H_n +  \Theta_n^{-1}
\over \sqrt{   \Theta_n} } \Big\} .
   \end{eqnarray*}

\section{\bf Proof of Corollary \ref{ger3}}
 By using the    generalization of Esseen's inequality  given in
 \cite{[P]}, Theorem 5 p.112, we have
\begin{eqnarray}\label{esseen}\sup_{x\in \R} \big|\P  \{T_n<x \} - \P\{g<x\} \big|&\le &
\frac{C_{{\rm E}}  }{   \psi (\sqrt {{\rm Var}(S_n')})}  \sum_{j=1}^n\E \psi (\overline \xi_j )   .
\end{eqnarray}
And the constant $C_{{\rm E}}$ is  numerical. Let
$\xi_j=\E_{L} X_j= V_j + (D/2) \e_j$, $\overline \xi_j=  \xi_j -\E_{(V,\e)}\xi_j$. By assumption $\psi(x)$ is convex and $\frac{x^3}{\psi(x)}$  is non-decreasing on $\R^+$. Thus
$\psi(ax)\ge a^3\psi(x)$ for
$0\le a\le 1$,
$x\ge 0$.     By Young's inequality,
$$\E \psi(2\xi_j)=\E_{(V,\e)}\psi(2 \E_{L} X_j) \le \E \psi(2X_j) .$$
Thus
\begin{eqnarray*}\E \psi(\overline\xi_j )&\le &\frac{1}{2}\big(\E  \psi(2\xi_j ) + \E  \psi( 2\E_{(V,\e)}\xi_j )\big) \le
\frac{1}{2}\big(\E \psi(2X_j) + \E \psi(2X_j)\big)
\cr &\le & \frac{1}{2}\big(8\E \psi( X_j)
+ 8\E \psi( X_j)\big) =8\E \psi( X_j ).
\end{eqnarray*}
 By reporting into (\ref{esseen})  we get
\begin{eqnarray}\label{esseen1}H_n  &\le &2^{  3/2}C_{{\rm E}}L_n
  \end{eqnarray}
 recalling that
$$L_n=\frac{  \sum_{j=1}^n\E \psi (X_j)  }
{   \psi (\sqrt
{ {\rm Var}(S_n )})}  .$$
The conclusion then follows directly from Corollary \ref{ger2}.


\section{Gamkrelidze's Local limit theorem.}


 We indicate in  this section how to recover Gamkrelidze's local limit theorem
 with an effective bound. To this extent we restate Lemma
1 of
\cite{MD} for the particular case of
$S_n=
\sum_{k=1}^n X_k$, where $X_k$ are integer--valued and independent. We prove it in greater detail than in the original paper. Let $(a_n)$ and $(b_n)$ two
sequences of real numbers, with $b_n >0$ for every $n$. We denote
 \begin{equation*}\label{CLT}
\rho_n:=  \sup_{p,q:p<q}\Big|\P\{p\le  S_n \le  q\} - \frac{1}{\sqrt{2\pi}}\int_{\frac{p-1-a_n}{ \sqrt{b_n}}}^{\frac{q-a_n}{
\sqrt{b_n} }}e^{\frac{-t^2}{2}}\, dt\Big|.
 \end{equation*}
 First we remark that
 \begin{eqnarray*}& &\P\{p\le  S_n \le  q\} - \frac{1}{\sqrt{2\pi}}\int_{\frac{p-1-a_n}{ \sqrt{b_n} }}^{\frac{q-a_n}{\sqrt{b_n}
}}e^{\frac{-t^2}{2}}\, dt\cr &= &\sum_{h=p}^q\Big\{ P(S_n=h)- \frac{1}{\sqrt{2\pi}}
\int_{\frac{h-1-a_n}{\sqrt{b_n} }}^{\frac{h-a_n}{\sqrt
{b_n} }}e^{\frac{-t^2}{2}}\, dt\Big\}
=\sum_{h=p}^q d_{h,n},
\end{eqnarray*}
 where
  \begin{equation*}d_{h,n}:=\P\{S_n=h\}- \frac{1}{\sqrt{2\pi}}\int_{\frac{h-1-a_n}{\sqrt {b_n} }}^{\frac{h-a_n}{ \sqrt {b_n}
}}e^{-\frac{ t^2}{2}}\, dt.\end{equation*}

 \begin{proposition}
 Suppose that
 \begin{equation}\label{ipotesiMc}
\sup_{n \in \mathbb{N}} \  b_ n\Big(\sup_{k \in \mathbb{Z}} \big|\P\{S_n=k+1\}- \P\{S_n=k\}\big|\Big) =M<  \infty.
\end{equation}
 Then there exists a constant $C$ depending on $M$ only such that
 \begin{equation}\label{primaparte}
 \sup_{k\in \mathbb{Z}}\sqrt{b_n}\big|d_{k,n}\big|\le  C\sqrt{\rho_n}.
 \end{equation}
  As a consequence\begin{equation}\label{secondaparte}
 \sup_{k\in \mathbb{Z}}\Big |\sqrt { b_n}\P\{S_n=k\}-\frac{1}{\sqrt{2\pi}\sigma}e^{-\frac{(k-a_n)^2}{2b_n}}\Big |\le  C\sqrt{\rho_n} +
\frac{1}{\sqrt{2\pi e} \sqrt {b_n} }.\end{equation}
 \end{proposition}
 The value of $C$ is explicited in the course of the proof.\begin{proof} Put
 \begin{equation*}\ell_{k,n}:= \frac{1}{\sqrt{2\pi}}\int_{\frac{k-1-a_n}{\sqrt{b_n}}}^{\frac{k-a_n}{\sqrt{b_n}}}e^{\frac{-t^2}{2}}\, dt,\end{equation*} and observe that\begin{align*} \big|\ell_{k+1,n}
- \ell_{k,n}\big|=\frac{1}{\sqrt{2\pi}\sqrt{b_n}}\Big|e^{-\frac{\xi_k^2}{2}}-e^{-\frac{\eta_k^2}{2}}\Big|,\end{align*}
with $\frac{k-1-a_n}{ \sqrt{b_n} }\le  \xi_k \le \frac{k-a_n}{ \sqrt{b_n} } \le  \eta_k \le \frac{k+1-a_n}{ \sqrt{b_n} }
$; and by Lagrange's Theorem
\begin{equation*}
\Big|e^{-\frac{\xi_k^2}{2}}-e^{-\frac{\eta_k^2}{2}}\Big|=|\xi_k - \eta_k|\cdot \big|\theta_k e^{-\frac{\theta_k^2}{2}}\big|\le
\frac{2}{\sqrt{e b_ n}},
\end{equation*}
with $\xi_k\le  \theta_k\le  \eta_k$ and $ \sup_{z\in \mathbb{R}}\big|ze^{-\frac{z^2}{2}}\big|=e^{-1/2}$. Hence
\begin{equation}\label{stima}
 \big|\ell_{k+1,n}
- \ell_{k,n}\big|\le \Big( \sqrt{\frac{2}{ e\pi}}\Big)\frac{1} {b_n}.
\end{equation} 
 Now we write \begin{eqnarray*} \label{prima}
& &   d_{k,n}=\P\{S_n=k\}-\ell_{k,n}    \cr &=&\big\{\P\{S_n=k\}-
\P\{S_n=k+1\}\big\}+ \big\{\ell_{k+1,n} - \ell_{k,n}\big\} +
d_{k+1,n}  \cr &\le & \sup_{k \in \mathbb{Z}}\big|\P\{S_n=k+1\}-
\P\{S_n=k\}\big|+\sup_{k \in \mathbb{Z}}\big|\ell_{k+1,n} -
\ell_{k,n}\big|+ d_{k+1,n} \le   \frac{R}{b_n}+d_{k+1,n}.
\end{eqnarray*}
 where we denote  $$R:=M+\sqrt{\frac{2}{ e\pi}} .$$

Similarly we also have
\begin{eqnarray} \label{seconda}
    d_{k,n}&\le&
\frac{R}{b_n}+ d_{k-1,n} . \end{eqnarray} 
Using induction we
find \begin{equation*} d_{h,n}\le \frac{R(k-h)}{b_n}+d_{k,n} \qquad
h<k
\end{equation*} \begin{equation*}
d_{k,n}\le  \frac{R(k-h)}{b_n}+d_{h,n}  \qquad  h<k;
\end{equation*}
putting together we have found
\begin{equation}\label{insieme}
\big|d_{h,n}-d_{k,n} \big|\le  \frac{R|k-h|}{b_n}  \qquad \forall  h,k.
\end{equation}

\noindent
We show that for every $\delta>0 $, for every $n$ and for every $k$,
\begin{equation*}
4R \rho_n< \delta^2 \Longrightarrow \sqrt{b_n}|d_{k,n}|<\delta,
\end{equation*}
thus proving \eqref{primaparte} with $C= 2 \sqrt R$.
Assume the contrary, i.e. there exist $\delta>0$, an integer $k_0$ and a positive integer $n_0$ such that
$$ 4R\rho_{n_0}< \delta^2, \quad \hbox{but }\quad \sqrt {b_{n_0}}\cdot\big|d_{k_0,n_0}\big|\ge  \delta.$$
To fix ideas, assume that $\sqrt {b_{n_0}} \cdot d_{k_0,n_0}\ge \delta$ . Consider the set of integers
\begin{align*}
A&= \Big\{h \in \mathbb{Z}:\frac{R|k_0-h|}{b_{n_0}}\le  \frac{\delta}{2 \sqrt{b_{n_0}}}\Big\}= \Big\{h \in \mathbb{Z}:k_0-\frac{\delta
\sqrt{b_{n_0}}}{2R}\le  h \le  k_0+\frac{\delta \sqrt{b_{n_0}}}{2R}\Big\}.
\end{align*}
From
\begin{align*}
{\rm card}\Big([r-\alpha, r+\alpha]\cap\mathbb{Z}\Big)= 2\alpha +1 - 2\{\alpha\}\ge  \alpha, \qquad \hbox{($\alpha \in \mathbb{R}^+$ and
$r \in \mathbb{Z}$)}
\end{align*}
we get
\begin{equation}\label{cardinalita}
{\rm card}(A)\ge  \frac{\delta \sqrt{b_{n_0}}}{2R},
\end{equation}
and by \eqref{insieme}, for every $h \in A$
\begin{equation*}
\frac{\delta }{\sqrt{b_{n_0}}}\le  d_{k_0,n_0}\le \big|  d_{k_0,n_0}-d_{h,n_0}\big|+ d_{h,n_0} \le  \frac{R|k_0-h|}{b_{n_0}}+
d_{h,n_0}\le  \frac{\delta}{2 \sqrt{b_{n_0}}}+ d_{h,n_0},
\end{equation*}
which implies
\begin{equation} \label{minorazione} d_{h,n_0}\ge \frac{\delta}{2 \sqrt{b_{n_0}}}.\end{equation}
Hence, by \eqref{cardinalita} and \eqref{minorazione},
\begin{align*}
& 4R\rho_{n_0}= 4R\cdot \sup_{p,q: p<q}\Big|\sum_{h=p}^qd_{h,n_0}\Big|\ge  4R\cdot \Big|\sum_{h=p_0}^{q_0}d_{h,n_0}\Big|=4R\cdot
\Big(\sum_{h=p_0}^{q_0}d_{h,n_0}\Big)\\&= 4R\cdot\Big(\sum_{h\in A}d_{h,n_0}\Big) \ge   4R\cdot\frac{\delta}{2 \sqrt{b_{n_0}}}\cdot
card(A) \ge  4R\cdot\frac{\delta}{2 \sqrt{b_{n_0}}}\cdot \frac{\delta \sqrt{b_{n_0}}}{2R}=\delta^2,
\end{align*}
a contradiction. This proves \eqref{primaparte}. In order to prove \eqref{secondaparte} we write
(for a suitable $\xi_k \in (k-1,k)$)
\begin{align*}
&\Big|\sqrt {b_n}\P\{S_n=k)-\frac{1}{\sqrt{2\pi }}e^{-\frac{(k-a_n)^2}{2b_n}}\Big|\le \sqrt{b_ n}
|d_{k,n}|+\Bigg|\frac{\sqrt{b_n}}{\sqrt{2\pi}}\int_{\frac{k-1-a_n}{
\sqrt{b_n}}}^{\frac{k-a_n}{ \sqrt{b_n}}}e^{\frac{-t^2}{2}}\,
dt-\frac{1}{\sqrt{2\pi }}e^{-\frac{(k-a_n)^2}{2b_n}}\Bigg|\\&=\sqrt{b_n}|d_{k,n}|+\frac{1}{\sqrt{2\pi
}}\Bigg|e^{-\frac{(\xi_k-a_n)^2}{2b_n}}-e^{-\frac{(k-a_n)^2}{2b_n}}\Bigg|\le  \sqrt {b_n} |d_{k,n}|+\frac{1}{\sqrt{2\pi }}\cdot
\frac{|\xi_k-k|}{ \sqrt {b_n}}\sup_{z\in \mathbb{R}}\big|ze^{-\frac{z^2}{2}}\big|\\&\le \sqrt {b_n} |d_{k,n}|+\frac{1}{\sqrt{2\pi e}
\sqrt  {b_n}}.\end{align*}
\end{proof}

   Now we   estimate   $M$ in (\ref{ipotesiMc}) by using (\ref{dep01}) which we recall
\begin{eqnarray*}
& &\Big |\P\{S_n=k)-\E_{(V,\epsilon)}\Big[{\bf 1}_{A_n}\cdot \frac{2}{\sqrt{2\pi B_n}}e^{- \frac{(k-W_n -D\frac{B_n}{2})^2}{D^2
\frac{B_n}{2}}}\Big]\Big|\cr  &\le &    2e^{- \frac{h^2 \Theta_n}{2(1+h/3))}}+
\frac{C}{(1-h)^{\frac{3}{2}}\Theta_n^{3/2}}.
\end{eqnarray*}
Thus
 \begin{eqnarray*}
& & \Big|\P\{S_n=k)-\P\{S_n=k+1)\Big| \  \le  \  2e^{- \frac{h^2 \Theta_n}{2(1+h/3))}}+
\frac{2C}{(1-h)^{\frac{3}{2}}\Theta_n^{3/2}}+ \cr & &\quad \E_{(V,\epsilon)}\Big[{\bf 1}_{A_n}\cdot \frac{2}{\sqrt{2\pi B_n}}\Big\{e^{-
\frac{(k+1-W_n -D\frac{B_n}{2})^2}{D^2 \frac{B_n}{2}}}-e^{- \frac{(k-W_n -D\frac{B_n}{2})^2}{D^2 \frac{B_n}{2}}}\Big\}\Big]\Big| ,
\end{eqnarray*}
and recalling that on $A_n$ we have $(1-h)\Theta_n \le  B_n \le  (1+h)\Theta_n$ we obtain
\begin{eqnarray*}& &\Big|\E_{(V,\epsilon)}\Big[{\bf 1}_{A_n}\cdot \frac{2}{\sqrt{2\pi B_n}}\Big\{e^{- \frac{(k+1-W_n -D\frac{B_n}{2})^2}{D^2
\frac{B_n}{2}}}-e^{- \frac{(k-W_n -D\frac{B_n}{2})^2}{D^2
\frac{B_n}{2}}}\Big\}\Big]\Big|\cr &\le & \Big|\E_{(V,\epsilon)}\Big[{\bf 1}_{A_n}\cdot \frac{2}{\sqrt{2\pi B_n}}\cdot \frac{\sqrt 2}{D
\sqrt{e B_n}}\Big]\Big|\cr  & \le &  \frac{2}{\sqrt{\pi e}(1-h)\Theta_n}.
\end{eqnarray*}
We conclude that
\begin{eqnarray}
& &b_n\sup_{k \in \mathbb{Z}}\Big|\P\{S_n=k)-\P\{S_n=k+1)\Big|\cr &\le  &
2b_ne^{- \frac{h^2 \Theta_n}{2(1+h/3))}}+ \frac{2Cb_n}{(1-h)^{\frac{3}{2}}\Theta_n^{3/2}}+\frac{2b_n}{\sqrt{\pi e}(1-h)\Theta_n},
\end{eqnarray}
which is bounded if we assume that
\begin{equation}\label{ipotesiM}
\limsup_{n \in \mathbb{N}}\frac{b_n}{\Theta_n}<  \infty.
\end{equation}
In particular,    in the case $b_n= {\rm Var} (S_n)$,   assumption \eqref{ipotesiM} is exactly  assumption
(iii) in  Remark \ref{123}.  For \eqref{ipotesiM} to hold, it suffices to assume that
 $$\inf_j \vartheta_{X_j} >0, \qquad \sup_j{\rm
Var}(X_j) <
\infty. $$
 
\begin{remark}  Assume that we have an effective bound for $\rho_n$ (as it happens with the Berry--Esseen theorems); in such a case from
\eqref{secondaparte} we automatically get an effective bound for
\begin{equation*}
 \sup_{k\in \mathbb{Z}}\Bigg|\sqrt { b_n}\P\{S_n=k)-\frac{1}{\sqrt{2\pi}\sigma}e^{-\frac{(k-a_n)^2}{2b_n}}\Bigg|.\end{equation*}
\end{remark}

 \section{Application to Random Walks in Random Scenery.}
Let   $X=\{ X_j, j\ge 1\}$ be a sequence  of i.i.d. square integrable random variables taking values in a   lattice $\mathcal L(v_{
0},D )$. Suppose we are given another sequence  $U=\{ U_j, j\ge 1\}$ of
integer--valued random variables, independent from
$X$.  We form the sequence of composed sums
$$S=\{  S_n, n\ge 1\}\qq {\rm where } \qq S_n =\sum_{k=1}^n X_{U_k} .$$
This defines a random walk in a  random scenery, this one being described by the sequence $U$.
 We  establish an effective Local Limit Theorem for  the sequence $S$. In a first step, we prove the analog of Theorem \ref{ger1} for the sequence $S$. Next we find a reasonable condition (see \eqref{ipotesir}) under which  Berry-Esseen's estimate is applicable. This is due to the surprising fact that under this condition, the intermediate conditioned sums in the Bernoulli part construction, are sums of  {\it i.i.d.} random variables. 
 
\subsection{Preliminary calculations.} By Lemma \ref{lemd},
 $   \{X_j, 1\le j\le n\} \buildrel{\mathcal D}\over{=}  \{V_j  +  \e_jL_j, 1\le j\le n\} $
where the random variables
$ (V_j,\e_j),L_j$, $j=1,\ldots,n
 $    are mutually independent and $\e_j$, $ L_j $ are
independent Bernoulli random variables with $\P\{\e_j=1\}= 1-\P\{\e_j=0\}=\t_j $ and $\P\{L_j =0\}=\P\{L_j=1\}=1/2$.  We thus  denote again $X_j= V_j+D\e_jL_j$ $1\le j\le n$.
The Corollary below is thus straightforward. Put,
$$W_n= \sum_{k=1}^n V_{U_k}, \quad M_n= \sum_{k=1}^n\varepsilon_{U_k} L_{U_k}, \quad B_n= \sum_{k=1}^n\varepsilon_{U_k}.$$
\begin{corollary}
For every $n\geqslant 1$ we have the representation
$$\{S_k, \, 1 \leqslant k \leqslant n\} \mathop{=}^\mathcal{D}\{W_k+DM_k, \, 1 \leqslant k \leqslant n\}.$$
\end{corollary}
\begin{remark}[Local time] We also have that
 $S_n=   \sum_{j=1}^\infty X_j \nu_n(j) $,
where $\nu_n(j)$ is the local time of the sequence $(U_j)$, i.e.
\begin{eqnarray*}\nu_n(j)&=&\begin{cases}0 &\quad {\rm if} \   U_k\not=j , 1\le k\le n\cr 
\#\big\{k; 1\le k\le n:
U_k=j\big\}&  \quad {\rm otherwise}.\end{cases} 
\end{eqnarray*}
 And so
$S_n=  \sum_{j=1}^\infty  (V_j + \varepsilon_j D L_j) \nu_n(j) \buildrel{\mathcal D}\over{=}\sum_{k=1}^n V_{U_k}+D\sum_{k=1}^n \varepsilon_{U_k}L_{U_k}$. However, we will not use properties of local time as it is standard for proving strong laws or local limit theorem. 

At this regard, our approach is new in the context of random scenery. We will still use the  Bernoulli part extraction approach, in developing more the algebra inherent to that construction, which in the  setting of random scenery reveals richer than expected.
\end{remark}

In what follows, we note    $V=\{V_j, j\ge  1\}$, $\varepsilon=\{\varepsilon_j, j\ge  1\}$,
$L=\{L_j, j\ge  1\}$.

\begin{lemma}
For every $k$, $\varepsilon_{U_k}$ is a Bernoulli random variable such that
$$\P\{\varepsilon_{U_k}=1\}= \E \vartheta_{U_k} .$$
Moreover, for $h \not =k$ we have
$$\P\{\varepsilon_{U_h}=1,\varepsilon_{U_k}=1\}=\E \vartheta_{U_h}\vartheta_{U_k} +\sum_{r=1}^\infty(\vartheta_r - \vartheta^2_r)\P\{U_h=r, U_k=r\}.$$ 
\end{lemma}
\begin{proof} By   independence of $U$ and $\varepsilon$ we have
\begin{eqnarray*}& &
\P\{\varepsilon_{U_k}=1\}= \sum_{r=1}^\infty \P\{\varepsilon_{U_k}=1,U_k=r \}=\sum_{r=1}^\infty \P\{\varepsilon_{r}=1,U_k=r \}\\&=&\sum_{r=1}^\infty
\P\{\varepsilon_{r}=1\}\P\{U_k=r \}=\sum_{r=1}^\infty \vartheta_{r} \P\{U_k=r \}= \E \vartheta_{U_k} .\end{eqnarray*} And similarly, using also the
independence of the variables  $\{\varepsilon_j, j\ge 1\}$,
\begin{align*}&
\P\{\varepsilon_{U_h}=1,\varepsilon_{U_k}=1\}\\&= \sum_{r,s=1}^\infty \P\{\varepsilon_{U_h}=1,\varepsilon_{U_k}=1,U_h=r,U_k=s \}=\sum_{r,s=1}^\infty
\P\{\varepsilon_{r}=1,\varepsilon_{s}=1, U_h=r,U_k=s  \}
\\&=\sum_{r =1}^\infty\P\{\varepsilon_{r}=1\}\P\{U_h=r,U_k=r \}+\sum_{r,s=1\atop r\neq s}^\infty \P\{\varepsilon_{r}=1\}\P\{\varepsilon_{s}=1\}\P\{U_h=r,U_k=s
\}\\&=\sum_{r}\vartheta_r\P\{U_h=r,U_k=r \}+\sum_{r,s=1\atop r\neq s}^\infty \vartheta_{r} \vartheta_{s} \P\{
U_h=r,U_k=s\}\\&=\sum_{r=1}^\infty(\vartheta_r-\vartheta^2_r\}\P\{U_h=r,U_k=r
\}+\sum_{r,s=1}^\infty \vartheta_{r} \vartheta_{s} \P\{ U_h=r,U_k=s\}\\&=
\sum_{r=1}^\infty(\vartheta_r-\vartheta^2_r)\P\{U_h=r,U_k=r \}+\E \vartheta_{U_h}\vartheta_{U_k} .\end{align*}
\end{proof}

Let
$$S_n^\prime= W_n+D \frac{B_n}{2}, \qq \quad n=1,2,\ldots.$$
The  two following Lemmas generalize Lemma 3.1
\begin{lemma} We have
$$ \E S_n = \E {S_n^\prime} .$$
\end{lemma}
 \begin{proof} Just observe that
\begin{eqnarray*} 
& &\E M_n = \sum_{k=1}^n \E\varepsilon_{U_k}L_{U_k}  =\sum_{k=1}^n\sum_{r=1}^\infty \E
\varepsilon_{r}L_{r}{\bf 1}_{\{U_k=r\}} \cr & &=\sum_{k=1}^n\sum_{r=1}^\infty \E
\varepsilon_{r}\E L_{r}\P\{ U_k=r\} = \sum_{k=1}^n\sum_{r=1}^\infty \frac{\vartheta_r}{2}\P\{ U_k=r\}
 =\frac{1}{2}\sum_{k=1}^n\E \vartheta_{U_k} =\E \frac{B_n}{2} .\end{eqnarray*}
\end{proof}

\noindent
Let
$$\Theta_n =\sum_{j=1}^n\E \vartheta_{U_j} .$$

\begin{lemma}\label{lemmagenerale}
We have
$$\E {S_n}^2 = \E({S_n^\prime})^2 +\frac{D^2\Theta_n }{4}+\frac{D^2}{4}\sum_{1\leqslant h,k \leqslant n\atop h \not= k}c_{h,k},$$
where
$$c_{h,k}= \sum_{r=1}^\infty \Big(\frac{3\vartheta^2_r}{4}-\frac{\vartheta_r}{2}\Big) \P\{U_h=r, U_k=r\}.$$ 
 \end{lemma}
\begin{proof}First
\begin{align}&\label{calcolaccio1}
\E   {S_n}^2 =\E \Big(W_n+ D \sum_{k=1}^n \varepsilon_{U_k}L_{U_k}\Big)^2   =\E W_n^2 +2D\,  \E\Big[W_n\big(\sum_{k=1}^n
\varepsilon_{U_k}L_{U_k}\big)\Big]+D^2\E \Big(\sum_{k=1}^n \varepsilon_{U_k}L_{U_k}\Big)^2 .\end{align}
 Now
\begin{align}\label{calcolaccio2}&\nonumber \E W_n\big(\sum_{k=1}^n \varepsilon_{U_k}L_{U_k}\big)  =\sum_{k=1}^n \E  W_n\varepsilon_{U_k}
L_{U_k}  =\sum_{k=1}^n \E\Big\{
\big(\sum_{h=1}^nV_{U_h}\big)\varepsilon_{U_k} L_{U_k}\Big\} \\&=\sum_{h,k=1}^n \E\big[ V_{U_h}\varepsilon_{U_k} L_{U_k}\big]
=\sum_{k=1}^n \E\big(
V_{U_k}\varepsilon_{U_k} L_{U_k}\big) +\sum_{h\not=k=1}^n \E\big( V_{U_h}\varepsilon_{U_k} L_{U_k}\big).\end{align} 
By   of $U$ with $(V,\varepsilon,
L)$, and    independence of $L$ with $(V, \varepsilon)$ we have
\begin{align*}&
\E\big( V_{U_k}\varepsilon_{U_k} L_{U_k}\big)=\sum_{r=1}^\infty\E\big( V_{U_k}\varepsilon_{U_k} L_{U_k}{\bf 1}_{\{U_k=r\}}\big) =\sum_{r=1}^\infty\E\big[
V_{r}\varepsilon_{r} L_{r}\big)\P\{U_k=r\}
\\& =\sum_{r=1}^\infty\E\big( V_{r}\varepsilon_{r}\big)\E\big( L_{r}\big)\P\{U_k=r\}=\frac{1}{2}\sum_{r=1}^\infty\E\big(
V_{r}\varepsilon_{r}
\big)\P\{U_k=r\} =\frac{1}{2}\sum_{r=1}^\infty\E\big(V_{r}\varepsilon_{r}{\bf 1}_{\{U_k=r\}} \big)\\&=  \frac{1}{2}\E\big( V_{U_k}\varepsilon_{U_k}
\big).\end{align*} In a similar way we have also
$$ \E\big( V_{U_h}\varepsilon_{U_k} L_{U_k}\big)= \frac{1}{2}\E\big( V_{U_h}\varepsilon_{U_k} \big), \qquad h \not =k.$$
Hence, continuing from \eqref{calcolaccio2}, we obtain
\begin{align}& \label{calcolaccio3}\nonumber \E\Big(W_n\big(\sum_{k=1}^n \varepsilon_{U_k}L_{U_k}\big)\Big)=\frac{1}{2}\Big(\sum_{k=1}^n \E\big(
V_{U_k}\varepsilon_{U_k} \big) +\sum_{1\leqslant h\not=k\leqslant n} \E\big( V_{U_h}\varepsilon_{U_k}
\big)\Big)\\&\nonumber=\nonumber\frac{1}{2}\Big(\sum_{k=1}^n \E\big( V_{U_k}\varepsilon_{U_k} \big) +\sum_{k=1}^n \sum_{h\not=k} \E\big( V_{U_h}\varepsilon_{U_k}
\big)\Big)=\frac{1}{2}\E\big(\sum_{k=1}^n V_{U_k}\varepsilon_{U_k}+\sum_{h\not=k}V_{U_h}\varepsilon_{U_k}\big)
\\&=\frac{1}{2}\E\Big(\sum_{k=1}^n\varepsilon_{U_k}\big(V_{U_k}+
\sum_{h\not=k}V_{U_h}\big)\Big)=\frac{1}{2}\E\Big(\big(\sum_{k=1}^n\varepsilon_{U_k}\big)W_n\Big)
=\E\big(\frac{B_n}{2} W_n\big).\end{align}
Lastly, 
\begin{align}\label{calcolaccio4}& \nonumber \E \big(\sum_{k=1}^n \varepsilon_{U_k}L_{U_k}\big)^2 =\sum_{k=1}^n
\E\big(\varepsilon^2_{U_k}L^2_{U_k}\big)+\sum_{1\leqslant h\not= k \leqslant n}\E\big(\varepsilon_{U_h}L_{U_h}\varepsilon_{U_k}L_{U_k}\big)
\\&=\sum_{k=1}^n \E\big(\varepsilon_{U_k}L_{U_k}\big)+\sum_{1\leqslant h \not= k \leqslant
n}\E\big(\varepsilon_{U_h}\varepsilon_{U_k}L_{U_h}L_{U_k}\big).\end{align}
 And
\begin{align}\label{calcolaccio5}&  \E\big(\varepsilon_{U_k}L_{U_k}\big)=
\sum_{r=1}^\infty\E\big(\varepsilon_{U_k}L_{U_k}{\bf 1}_{\{U_k=r\}}\big)=
\sum_{r=1}^\infty\E\big(\varepsilon_{r}L_{r}{\bf 1}_{\{U_k=r\}}\big)\\& \nonumber
 =\sum_{r=1}^\infty\E\big(\varepsilon
_{r}\big)\E\big(L_{r}\big)\P\{U_k=r\}=\frac{1}{2} \sum_{r=1}^\infty\E\big(\vartheta
_{r}\big) \P\{U_k=r\} =\frac{1}{2}\, \E \vartheta
_{U_k} .\end{align}
Similarly
\begin{align}\label{calcolaccio6}& \nonumber
\E \varepsilon_{U_h}\varepsilon_{U_k}L_{U_h}L_{U_k} =
\sum_{r,s=1}^\infty\E\big(\varepsilon_{U_h}\varepsilon_{U_k}L_{U_h}L_{U_k}{\bf 1}_{\{U_h=r,
U_k=s\}}\big)=\sum_{r,s=1}^\infty\E\big(\varepsilon_{r}\varepsilon_{s}L_{r}L_{s}{\bf 1}_{\{U_h=r,
U_k=s\}}\big)\\&=\nonumber\sum_{r=1}^\infty\E \varepsilon_{r} \E L_{r}  \P\{U_h=r,
U_k=r\} +\sum_{r,s=1\atop r\neq s}^\infty\E \varepsilon_{r} \E \varepsilon_{s}
 \E L_{r} \E L_{s} \P\{U_h=r, U_k=s\} 
\\&\nonumber =\frac{1}{2}\sum_{r =1 }^\infty\vartheta_r \P\{U_h=r,
U_k=r\}+\frac{1}{4}\sum_{r,s=1\atop r\neq s}^\infty\vartheta_{r}\vartheta_{s}\P\{U_h=r, U_k=s\}\\&=\nonumber\sum_{r =1
}^\infty\Big(\frac{\vartheta_r}{2}-\frac{\vartheta^2_r}{4}\Big) \P\{U_h=r,
U_k=r\}+\frac{1}{4}\sum_{r,s=1}^\infty\vartheta_{r}\vartheta_{s}\P\{U_h=r, U_k=s\}
\\&=\frac{1}{4}\E \vartheta_{U_h}\vartheta_{U_k} +\sum_{r=1}^\infty\Big(\frac{\vartheta_r}{2}-\frac{\vartheta^2_r}{4}\Big) \P\{U_h=r,
U_k=r\}=\frac{1}{4}\E \vartheta_{U_h}\vartheta_{U_k} +a_{h,k},\end{align} 
where we set
$$a_{h,k}=\sum_{r=1}^\infty\Big(\frac{\vartheta_r}{2}-\frac{\vartheta^2_r}{4}\Big) \P\{U_h=r, U_k=r\}.$$
Then, by inserting \eqref{calcolaccio5} and \eqref{calcolaccio6} into \eqref{calcolaccio4} we get \begin{align}\label{calcolaccio7}& \nonumber
\E \Big(\sum_{k=1}^n \varepsilon_{U_k}L_{U_k}\Big)^2 =\frac{1}{2}\sum_{k=1}^n\E\vartheta _{U_k} +\sum_{1\leqslant h \not= k \leqslant
n}\Big(\frac{1}{4}\E \vartheta_{U_h}\vartheta_{U_k} +a_{h,k}\Big)
\\&=\nonumber\frac{\Theta_n}{2}+\frac{1}{4} \sum_{1\leqslant h , k \leqslant
n}\E \vartheta_{U_h}\vartheta_{U_k} -\frac{1}{4} \sum_{k=1}^n\E \vartheta^2 _{U_k} +\sum_{1\leqslant h \not= k \leqslant
n}a_{h,k}\\&=\frac{\Theta_n}{2}+\frac{1}{4}\E	\Big\{\Big(\sum_{k=1}^n\vartheta _{U_k}\Big)^2-	\sum_{k=1}^n\vartheta^2 _{U_k}\Big\}+\sum_{1\leqslant h \not= k
\leqslant n}a_{h,k}.\end{align} 

Now, inserting \eqref{calcolaccio3} and \eqref{calcolaccio7} in \eqref{calcolaccio1}  we find
\begin{align*}&
\E  {S_n}^2 =\E W_n^2 +2D  \E\Big(W_n \sum_{k=1}^n \varepsilon_{U_k}L_{U_k}\Big) +D^2\E \Big(\sum_{k=1}^n
\varepsilon_{U_k}L_{U_k}\Big)^2 \\&=\E W_n^2 +2D\E\Big(\frac{B_n}{2} W_n\Big)+\frac{D^2}{2}\Theta_n+\frac{D^2}{4}\E\Big\{	\Big(\sum_{k=1}^n\vartheta
_{U_k}\Big)^2-	\sum_{k=1}^n\vartheta^2 _{U_k}\Big\}+\frac{D^2}{4}\sum_{1\leqslant h \not= k \leqslant n}a_{h,k}.
\end{align*}
On the other hand
\begin{align*}&\E W_n^2 +2D\E\Big(\frac{B_n}{2} W_n\Big)=\E \Big(W_n+D\frac{B_n}{2}\Big)^2 -\frac{D^2}{4}\E B_n^2 =
\E (S^\prime_n)^2 -\frac{D^2}{4}\E B_n^2.\end{align*} 

Hence
\begin{align}&\label{calcolaccio8}
\E {S_n}^2=\E(S^\prime_n)^2 -\frac{D^2}{4}\E B_n^2+\frac{D^2}{2}\Theta_n+\frac{D^2}{4}\E\Big\{	\Big(\sum_{k=1}^n\vartheta
_{U_k}\Big)^2-	\sum_{k=1}^n\vartheta^2
_{U_k}\Big\}+\frac{D^2}{4}\sum_{1\leqslant h \not= k \leqslant n}a_{h,k}.\end{align}
Now, in a similar way as we did for  \eqref{calcolaccio7}, we find that
\begin{align*}&\E\, B_n^2=\E\Big(
\sum_{k=1}^n\varepsilon_{U_k}\Big)^2=\Theta_n+\E\Big\{	\Big(\sum_{k=1}^n\vartheta
_{U_k}\Big)^2-	\sum_{k=1}^n\vartheta^2
_{U_k}\Big\}+\sum_{1\leqslant h \not= k \leqslant n}b_{h,k},\end{align*}
where
$$b_{h,k}=\sum_{r=1}^\infty\big(\vartheta_r-\vartheta^2_r\big)\, \P\{U_h=r, U_k=r\}.$$
and inserting into \eqref{calcolaccio8}, we obtain
\begin{align*}&\E  {S_n}^2 =\E (S^\prime_n)^2  -\frac{D^2}{4}\Big\{\Theta_n+\E\Big\{	\Big(\sum_{k=1}^n\vartheta
_{U_k}\Big)^2-	\sum_{k=1}^n\vartheta^2
_{U_k}\Big\}+\sum_{1\leqslant h \not= k \leqslant n}b_{h,k}\Big\}\\&+\frac{D^2}{2}\Theta_n+\frac{D^2}{4}\E\Big\{	\Big(\sum_{k=1}^n\vartheta
_{U_k}\Big)^2-	\sum_{k=1}^n\vartheta^2
_{U_k}\Big\}+\frac{D^2}{4}\sum_{1\leqslant h \not= k \leqslant
n}a_{h,k}\\&=\E  (S^\prime_n)^2 +\frac{D^2\Theta_n}{4}+\frac{D^2}{4} \sum_{1\leqslant h \not= k \leqslant
n}(a_{h,k}-b_{h,k}) =\E (S^\prime_n)^2 +\frac{D^2\Theta_n}{4}+\frac{D^2}{4}\sum_{1\leqslant h \not= k \leqslant n}c_{h,k},\end{align*} 
where
$$c_{h,k}=a_{h,k}-b_{h,k}= \sum_{r=1}^\infty\Big(\frac{3\vartheta^2_r}{4}-\frac{\vartheta_r}{2}\Big) \P\{U_h=r, U_k=r\}.$$
\end{proof}\begin{remark}\label{ipotesi}(i)
Assume that the variables $(U_j)$  verify
\begin{equation}\label{ipotesir}\P\{U_h=r, U_k=r\}=0, \qquad \forall \, h \not =k \quad \hbox{and }\forall \, r.\end{equation}
Then  from Lemma \ref{lemmagenerale} we get
$$\E  {S_n}^2 = \E ({S_n^\prime})^2 +\frac{D^2\Theta_n }{4}.$$
The above assumption holds for instance in the following important case: let the $U_j$ be the partial sums of a sequence of  random variables $(Y_i)$ taking positive integer values $$U_j= \sum_{i=1}^j Y_i.$$
This is the case if $Y_i\equiv 1$ for every $i$, so that $U_j=j$ for every $j$. Hence our present discussion is a generalization of the previous one.

(ii) Let the $U_j$ be the partial sums of a sequence of   independent  random variables $(Y_i)$. Then, for $h < k$,
\begin{eqnarray*}
 \P\big\{U_h=r, U_k=r\big\}&=&\P \{ U_h=r\}\P\Big\{\sum_{i=h+1}^kY_i=0 \Big\}.\end{eqnarray*}
Hence
\begin{align*}&c_{h,k}= \rho_{h,k}\sum_{r=1}^\infty\Big(\frac{3\vartheta^2_r}{4}-\frac{\vartheta_r}{2}\Big) \P\{U_h=r\}=\P\Big\{\sum_{i=h+1}^kY_i=0
\Big\}\E\Big(\frac{3\vartheta^2_{U_h}}{4}-\frac{\vartheta_{U_h}}{2}\Big).\end{align*} 
Notice that, if the random variables $(Y_i)$ are i.i.d, then
\begin{align*}&\P\Big\{\sum_{i=h+1}^kY_i=0 \Big\}=\P\Big\{\sum_{i=1}^{k-h}Y_i=0 \Big\}= \sigma_{k-h},\end{align*}
where $\sigma_{n}=\P\{U_n=0\}$.

\end{remark}\subsection{The Local Limit Theorem with effective rate}
In this section we keep all the notations of the preceding one; furthermore we set
$$H_n= \sup_{x \in \mathbb{R}}\Big|P\Big(\frac{S^\prime_n- \E[S^\prime_n]}{\sqrt{Var(S^\prime_n)}}<x\Big)-\Phi(x)\Big|$$
$$\rho_n=P\Big(\Big|\sum_{k=1}^n\varepsilon_{U_k}- \Theta_n\Big|> h\Theta_n \Big),$$
where $\Phi$ is the distribution function of the standard gaussian law. 
\vskip 3 pt
The following Theorem now generalizes Theorem \ref{ger1} for the case of random scenery. Its proof is identical to that of Theorem \ref{ger1} (just replace $\vartheta_k$ with $\vartheta_{U_k}$ in each formula of Theorem  \ref{ger1}), so we omit it.
\begin{theorem}\label{LLTRS}
For any $0<h<1 $, $0<\vartheta_j \leqslant \vartheta_{X_j}$ and all $\kappa \in \mathcal{L}(v_0n, D)$\begin{align*}&
\P\{S_n=\kappa\}\leqslant \Big(\frac{1+h}{1-h}\Big)\frac{D}{\sqrt{2 \pi Var(S_n)}}e^{-\frac{(\kappa-E[S_n])^2}{2(1+h){\rm Var}(S_n)}}\\&+\frac{C_1}{\sqrt{(1-h)}\Theta_n}\Big(H_n+ \frac{1}{(1-h)\Theta_n}\Big)+\rho_n(h);\end{align*}
\begin{align*}&
\P\{S_n=\kappa\}\geqslant \Big(\frac{1-h}{1+h}\Big)\frac{D}{\sqrt{2 \pi Var(S_n)}}e^{-\frac{(\kappa-E[S_n])^2}{2(1-h){\rm Var}(S_n)}}\\&-\frac{C_1}{\sqrt{(1-h)\Theta_n}}\Big(H_n+ \frac{1}{(1-h)\Theta_n}+2\rho_n(h)\Big)-\rho_n(h).
\end{align*}

\end{theorem}

\subsection{Covariance structure of the sequence $\boldsymbol{\big\{V_{U_k}+\frac{D}{2}\varepsilon_{U_k},k\ge 1\big\}}$.}
Denote
$$Y_k=V_{U_k}+\frac{D}{2}\varepsilon_{U_k}.$$We observe that
$$S^\prime_n = W_n +\frac{D}{2}B_n = \sum_{k=1}^n\Big(V_{U_k}+\frac{D}{2}\varepsilon_{U_k}\Big)=\sum_{k=1}^nY_k$$
and that the quantity $H_n$ appearing in the statement of Theorem
\ref{LLTRS} concerns precisely the sequence of partial sums $S^\prime_n$.
The aim of the present section is to discuss suitable assumptions assuring the  independence  of the variables $\{Y_k, k\ge 1\}$, thus enabling us to give an estimation of \lq\lq Berry--Esseen type\rq\rq \ for $H_n$. 

\vskip 2 pt
Throughout this section we assume that the variables $\{U_j, j\ge 1\}$  verify
 condition \eqref{ipotesir} appeared in Remark \ref{ipotesi} (i), i.e.  $$ \P\big\{U_h=r, U_k=r\big\}=0, \qquad\quad  \forall \, h \not =k \quad \hbox{and}\quad \forall \, r\ge 1. $$ 
\begin{theorem}
Let the $\{X_n, n\ge 1\}$ be i.i.d. Assume moreover that, for every pair $(h,k)$ with $h \not =k$, the random variables $\vartheta_{U_h}$ and $\vartheta_{U_k}$ are uncorrelated.  Then the sequence $\{Y_k, k\ge 1\}$ is i.i.d.\end{theorem}
\begin{remark}
The assumption of the above theorem is valid if either
\begin{itemize}
\item[(i)] $r \mapsto \vartheta_r$
is constant (for instance $\vartheta_r = \vartheta_{X_r}=\vartheta_{X} $ for every $r$),
\vskip 1 pt \item[(ii)] $U_h$ and $U_k$ are independent (and trivially if  $U_h= h$, for every $h$).
\end{itemize}
\end{remark}
Let $\phi: \mathbb{R}\to \mathbb{R}$ be a measurable function and denote
\begin{equation*}\label{definizione}
\Delta \phi (t) = \phi\Big( t+\frac{D}{2}\Big)-\phi\big( t \big).
\end{equation*}
The above theorem is a straightforward consequence of 
\begin{proposition} Let the sequence $\{X_n, n\ge 1\}$  be i.i.d. Then, for every pair $\phi,\psi $ of measurable functions $ \mathbb{R}\to \mathbb{R}$,
$$\E\big[\phi(Y_{h})\psi(Y_{k})\big]= \E\big[(\alpha_{\phi}+ \beta_{\phi}\vartheta_{U_{h}})(\alpha_{\psi}+ \beta_{\psi}\vartheta_{U_{k}})\big],\qquad h \not =k$$
where
$$\alpha_\phi=\E\phi(X_1)=\sum_{k=1}^\infty f(k)\phi(v_k) ,\qquad 
\beta_\phi = -\frac{1}{2}\sum_{k=1}^\infty \frac{f(k)\wedge f(k+1)}{\vartheta_{X}} \Delta^2 \phi(v_k) .$$$$\alpha_\psi=\E\psi(X_1)=\sum_{k=1}^\infty f(k)\psi(v_k) ,\qquad 
\beta_\psi = -\frac{1}{2}\sum_{k=1}^\infty \frac{f(k)\wedge f(k+1)}{\vartheta_{X}} \Delta^2 \psi(v_k) .$$
In particular, for every pair $A$ and $B$ of Borel subsets of $\R$,
\begin{equation*}
\P\{Y_h \in A, Y_k \in B\}-\P\{Y_h \in A\}\P\{Y_k \in B\}= {\rm Cov}({\bf 1}_A(Y_h),{\bf 1}_B(Y_k))=\beta_A \beta_B {\rm Cov} (\vartheta_{U_{h}},\vartheta_{U_{k}})
\end{equation*}
where
$$ \beta_A =\beta_{{\bf 1}_A},\qquad \beta_B =\beta_{{\bf 1}_B}.$$
 \end{proposition}
\begin{proof} Since the $X_r$ are identically distributed, we shall drop the symbol $r$ in the definition of $f_r$; moreover (see Section 1, before \eqref{basber0})
$$\tau_k^{(r)}=\vartheta_r  \frac{f(k)\wedge f(k+1)}{\vartheta_{X}}.$$ 

First, for every $r$,\begin{eqnarray*}&&
\E\phi\Big( V_{r}+\frac{D}{2}\varepsilon_{r}\Big)
\cr &=&
\sum_{k=1}^\infty \phi\Big( v_k+\frac{D}{2}\Big)\P\{V_{r}=v_k,\varepsilon_{r}=1\}+\sum_{k=1}^\infty \phi\big( v_k\big)\P\{V_{r}=v_k,\varepsilon_{r}=0\}
\cr &=&\sum_{k=1}^\infty \phi\Big( v_k+\frac{D}{2}\Big)\tau_k^{(r)}+\sum_{k=1}^\infty \phi\big( v_k\big)\Big(f(k) - \frac{\tau_{k-1}^{(r)}+\tau_k^{(r)}}{2}\Big)
\cr &=&\sum_{k=1}^\infty \phi\Big( v_k+\frac{D}{2}\Big)\tau_k^{(r)}+\sum_{k=1}^\infty \phi\big( v_k\big)f(k)-\frac{1}{2}\sum_{k=1}^\infty \phi\big( v_k\big)\tau_{k-1}^{(r)}-\frac{1}{2}\sum_{k=1}^\infty \phi\big( v_k\big)\tau_{k}^{(r)}
\cr &=&
\sum_{k=1}^\infty \phi\Big( v_k+\frac{D}{2}\Big)\tau_k^{(r)}+\sum_{k=1}^\infty \phi\big( v_k\big)f(k)-\frac{1}{2}\sum_{k=1}^\infty \phi\big( v_{k-1}+D\big)\tau_{k-1}^{(r)}-\frac{1}{2}\sum_{k=1}^\infty \phi\big( v_k \big)\tau_{k}^{(r)}
\cr &=&
\sum_{k=1}^\infty \phi\Big( v_k +\frac{D}{2}\Big)\tau_k^{(r)}+\sum_{k=1}^\infty \phi\big( v_k \big)f(k)-\frac{1}{2}\sum_{k=1}^\infty \phi\big( v_{k} +D\big)\tau_{k}^{(r)}-\frac{1}{2}\sum_{k=1}^\infty \phi\big( v_k \big)\tau_{k}^{(r)}\cr &=&
\sum_{k=1}^\infty\tau_{k}^{(r)}\Big\{\phi\big( v_k +\frac{D}{2}\big)-\frac{\phi\big( v_{k} +D\big)+\phi\big( v_{k} \big)}{2}\Big\}+\sum_{k=1}^\infty \phi\big( v_k \big)f(k)
\cr &=&\sum_{k=1}^\infty \phi\big( v_k \big)f(k)-\frac{1}{2}\sum_{k=1}^\infty\tau_{k}^{(r)}\Delta^2\phi(v_k)
\cr &=&\sum_{k=1}^\infty \phi\big( v_k \big)f(k)-\frac{\vartheta_r}{2}\sum_{k=1}^\infty \frac{f(k)\wedge f(k+1)}{\vartheta_{X}}\Delta^2\phi(v_k)=\alpha_\phi + \beta_\phi\vartheta_r.\end{eqnarray*}
Similarly,
\begin{align*}&\E\psi\Big( V_{s}+\frac{D}{2}\varepsilon_{s}\Big)=\alpha_\psi + \beta_\psi\vartheta_s.\end{align*}
Hence, observing that for $r\not =s$ the random variables $V_r +\frac{D}{2}\varepsilon_r$ and $V_s +\frac{D}{2}\varepsilon_s$ are independent, we have\begin{align*}&
\E\big[\phi(Y_{h})\psi(Y_{k})\big]=\sum_{r,s=1}^\infty\E\Big[\phi\big(V_r +\frac{D}{2}\varepsilon_r\big)\psi\big(V_s +\frac{D}{2}\varepsilon_s\big)\Big]\P\big\{U_h=r,U_k=s\big\}\\&=\sum_{r,s=1\atop r\not =s}^\infty\E\Big[\phi\big(V_r +\frac{D}{2}\varepsilon_r\big)\psi\big(V_s +\frac{D}{2}\varepsilon_s\big)\Big]\P\big\{U_h=r,U_k=s\big\}\\&=\sum_{r,s=1\atop r\not =s}^\infty\E\big[\phi\big(V_r +\frac{D}{2}\varepsilon_r\big)\big]\E\big[\psi\big(V_s +\frac{D}{2}\varepsilon_s\big)\big]\P\big\{U_h=r,U_k=s\big\}\\&=\sum_{r,s=1\atop r\not =s}^\infty
(\alpha_\phi + \beta_\phi\vartheta_r)(\alpha_\psi + \beta_\psi\vartheta_r)\P\big\{U_h=r,U_k=s\big\} \\&=\sum_{r,s=1}^\infty
(\alpha_\phi + \beta_\phi\vartheta_r)(\alpha_\psi + \beta_\psi\vartheta_r)\P\big\{U_h=r,U_k=s\big\}=\E(\alpha_{\phi}+ \beta_{\phi}\vartheta_{U_{h}})(\alpha_{\psi}+ \beta_{\psi}\vartheta_{U_{k}}).\end{align*}
\end{proof}
\begin{remark}
Let $A=[a,b]$ be a closed interval in $\mathbb{R}$. Let 
$$p= \max\{k: v_k < a\}, \qquad q = \max\{k: v_k \leqslant b\}.$$
It is easy to see that
$$-\frac{1}{2}\Delta^2 \phi(v_p)= \begin{cases}+\frac{1}{2}
&\hbox{\rm if }v_p +\frac{D}{2}\in A\\-\frac{1}{2}& \hbox{\rm if }v_p +\frac{D}{2}\not \in A;
\end{cases}
$$
similarly$$-\frac{1}{2}\Delta^2 \phi(v_q)= \begin{cases}+\frac{1}{2}
&\hbox{\rm if }v_q +\frac{D}{2}\in A\\-\frac{1}{2}& \hbox{\rm if }v_q +\frac{D}{2}\not \in A.
\end{cases}
$$
It follows that
$$\big|\beta_A\big| =\Big|-\frac{1}{2} \frac{f(p)\wedge f(p+1)}{\vartheta_{X}} \Delta^2 \phi(v_p)-\frac{1}{2} \frac{f(q)\wedge f(q+1)}{\vartheta_{X}} \Delta^2 \phi(v_q)\Big|\leqslant 1,$$
since
$$ \frac{f(k)\wedge f(k+1)}{\vartheta_{X}}\leqslant 1, \qquad \forall\,\, k.$$
As a consequence we get\begin{equation*}
\big|\P\{Y_h \in A, Y_k \in B\}-\P\{Y_h \in A\}\P\{Y_k \in B\}\big|\leqslant  \Big|{\rm Cov} (\vartheta_{U_{h}},\vartheta_{U_{k}})\Big|
\end{equation*}
A similar argument yields the above inequality for any interval in $\R$ (open, or half--closed, or unbounded).
\end{remark}

\section{\bf Concluding Remarks and Open Problems.}
We conclude  with discussing  two important  questions concerning the approach used. The first 
concerns moderate deviations, and the second is related to weighted sums. \subsection{\gsec Moderate deviation local limit theorems.}
In the i.i.d. case, the general form of the local limit theorem
(\cite{IL}, Th. 4.2.1) states
\begin{theorem}\label{th:G4}
  In order that  for some choice of constants $a_n$ and $b_n$
$$\lim_{n \to \infty}\sup_{N \in \mathcal{L}(v_0n, D)}\Big|\frac{b_n}{\lambda}\P\{S_n=N\}-g\big( \frac{N-a_n}{b_n}\big)\Big|=0, $$
where $g$ is the density of some stable distribution $G$ with exponent $0< \alpha \leq 2$,
 it is necessary and sufficient that
$$ {\rm (i)}\ \
 \frac{S_n-a_n}{b_n} \buildrel{\mathcal D}\over{\Rightarrow}   G   \ \
\hbox{as $n \to \infty$} \qq\qq
    {\rm (ii)}\ \   \hbox{$D $ is maximal}.$$
\end{theorem}
\noi This provides a useful estimate of $\P\{S_n=N\} $ for the values of
$N$ such that ${|N | /  b_n }$ is bounded,   as already mentioned
when $\a=2$ (with $b_n=\sqrt {\Sigma_n}$ using notation
(\ref{not1})). When ${|N | /  b_n }\to \infty$, it is known, at
least when $0<\a<1$, that another estimate exists. More precisely,
$$\P\{S_n=N\} \sim n\P\{X=N\} \qq \hbox{as $n \to \infty $,} $$
uniformly in $n$ such that ${|N | /  b_n }\to \infty$. We refer to Doney \cite{D} for   large deviation local limit theorems.
 In the intermediate range of values where ${|N | /  b_n }$ can be large but not too large with respect to
$n$, it was known  already three centuries ago  that in the binomial case  finer estimates are  available
for this range of values.  \begin{lemma}\label{moivre} {\rm (De Moivre--Laplace, 1730)} Let
$0<p<1$,
$q=1-p$. Let $X$ be such that
$\P\{X=1\}=p=1-\P\{X=0\}
$. Let
$X_1, X_2,\ldots$ be independent copies of
$X$ and let $S_n=X_1+\ldots +X_n$. Let   $0<\g<1$ and let
 $ \b \le \g\sqrt{  pq}\, n^{1/3} $. Then for all $k$ such that
letting  $ x= \frac{k-np}{\sqrt{  npq}} $, $|x|\le \b n^{1/6}$,
we have
 \begin{eqnarray*} \P\{S_n=k\}  &=&   \frac{e^{-  \frac{x^2}{  2  }} }{\sqrt{2\pi npq}} \   e^E ,\end{eqnarray*}
with
 $|E|\le  \frac{ |x|^3}{\sqrt{ npq}}+ \frac{|x|^4 }{npq}+ \frac{|x|^3}{2(npq)^{\frac{3}{2}} } +  \frac{1}{ 4n\min(p,q)(1 -  \g      )}$.
   \end{lemma}
See Chow and Teicher \cite{CT}.
    Although the uniform estimate given in Lemma \ref{lltber} is optimal (it is
derived from a fine local limit theorem with asymptotic expansion), it is   for a moderate deviation like $x\sim n^{1/7}$,   considerably
less precise than the old    one of De Moivre (case
$p=q$).
\vskip 2 pt
 {\gsec Problem I} Under which moment assumptions,
  De Moivre-Laplace's estimate extends to sums of independent random variables?
\vskip 1 pt
\noindent 
 A partial answer can be given by means of the following result  proved by Chen, Fang  and Shao \cite{CFS}.
\begin{theorem} Let $X_i$, $1 \leqslant i \leqslant n$ be a sequence
of independent random variables with $\E[X_i]=0$. Put $S_n=
 \sum_{i=1}^n X_i$ and $B_n^2 =\sum_{i=1}^n \E X_i^2  $. Assume that
 there exist  positive constant $c_1$, $c_2$ and $t_0$ such that
 $$B_n^2 \geqslant c_1^2 n,\qquad \E e^{t_0\sqrt{|X_i|}} \leqslant  c_2 \qq \qq  \hbox{for} \  1 \leqslant i \leqslant n.$$
Then $$\Big|\frac{\P\{S_n/B_n\geqslant x\}}{1-\Phi(x)}-1 \Big|\le 
c_3\frac{(1+x^3)}{\sqrt n},$$ for $0 \leqslant x \leqslant
(c_1t_0^2)^{1/3} n^{1/6}$,  where $c_3$  
depends on $c_2$ and $c_1t_0^2.$\end{theorem}

\noindent
Consider a sequence $X_i$ of integer--valued random variables and
assume for simplicity that $ c_3=1$. Let $k$ be an integer. The above result gives
\begin{eqnarray*}
 & & \P\{S_n=k\}= \P\big\{\frac{S_n}{B_n}\geqslant
\frac{k}{B_n}\big\}-\P\big\{\frac{S_n}{B_n}\geqslant
\frac{k+1}{B_n}\big\}
\cr &=&\big(1+\frac{1+ (\frac{k}{B_n} )^3}{\sqrt
n}\big) \big(1-\Phi\big(\frac{k}{B_n}\big)\big)-
 \big(1+
\frac{1+ (\frac{k+1}{B_n} )^3}{\sqrt
n}\big)\big(\big(1-\Phi\big(\frac{k}{B_n}\big)\big)+\Phi\big(\frac{k}{B_n}\big)-
\Phi\big(\frac{k+1}{B_n}\big)\big)
\cr &=&
\Big\{1+ \frac{1+  (\frac{k+1}{B_n} )^3}{\sqrt
n}+\frac{1}{\sqrt n}\,
\frac{\big(1-\Phi (\frac{k}{B_n} )\big)\big( (\frac{k}{B_n} )^3- (\frac{k+1}{B_n}  )^3\big)}{\Phi (\frac{k+1}{B_n} )-\Phi (\frac{k}{B_n
 } )} \ \Big\} 
 \big(\Phi  (\frac{k+1}{B_n
 } )-\Phi (\frac{k}{B_n} )\big).
\end{eqnarray*}
 Now
 \begin{eqnarray*}\Phi\big( ({k+1})/{B_n}\big)-\Phi\big( {k}/{B_n}\big)&\approx&
 \frac{1}{\sqrt{2\pi}B_n} e^{-\frac{k^2}{2B_n^2}}\cr 
  1-\Phi\big( {k}/{B_n}\big)&\approx & \frac {B_n}{k}e^{-\frac {k^2}{2B_n^2}}
\cr \big( {k}/{B_n}\big)^3-\big( {(k+1)}/{B_n}\big)^3&\approx &- {3k^2}/{B_n^3}.
 \end{eqnarray*}
Putting into the above expression we find the approximation
 \begin{eqnarray*}
  \P\{S_n=k\}&\approx& \Big\{1 + \frac{k^3}{\sqrt n B_n^3}-  \frac{B_n}{k
\sqrt n} \cdot\frac{3k^2}{B_n^3}\Big\} \frac{1}{\sqrt{2\pi}B_n}
e^{-\frac{k^2}{2B_n^2}}\cr &=&\Big\{1 + \frac{k^3}{\sqrt n B_n^3}-
\frac{3k}{\sqrt n B_n^2}\Big\} \frac{1}{\sqrt{2\pi}B_n}
e^{-\frac{k^2}{2B_n^2}}   \ =  \ e^E \frac{1}{\sqrt{2\pi}B_n}
e^{-\frac{k^2}{2B_n^2}},\end{eqnarray*} 
with  $E=\frac{k^3}{\sqrt n
B_n^3}+ \frac{3k}{\sqrt n B_n^2}$. 

However, assumption $\E e^{t_0\sqrt{|X_i|}} \le  c_2$ is  restrictive, since  the constant $c_2$ can be quite large. 
 Consider for instance the following  remarkable example. 
\vskip 2 pt 
\noi {\emph{Probabilistic  model of  the partition function:}} We refer to Freiman-Pitman \cite{FP}. 
 Let
$\s $ be a real. Fix some positive integer $n$, and let $1\le m\le n$.  
  Let $X_m, \ldots , X_n$ be independent random variables defined by 
$$ \P\{X_j =0\}= \frac{1}{1+e^{-\s j}}, \qq\qq \P\{X_j =j\}= \frac{e^{-\s j}}{1+e^{-\s j}}.$$
  The random variable $Y= X_m + \ldots + X_n$ can serve to modelize the partition function $q_m(n)$
 counting the number of partitions of $n$ into distinct parts, each of which is at least $m$, namely the number of ways to express $n$ as 
$$ n= i_1+ \ldots +i_r, \qq \qq m\le i_1<\ldots <i_r\le n.$$
(By Euler's penthagonal theorem, $q_0(n)$   for instance appears as  a coefficient in the expansion of $\prod_{k\le n}
(1+e^{ik\theta})$.)  
Notice that we have the following formula (in which  $\s$   only appears in the right-hand side)
\begin{equation}   \label{qmn} q_m(n) = e^{\s n} \int_0^1 \prod_{j=m}^n \big(1 + e^{-\s j} e^{2i\pi \a j}\big)e^{-2i\pi \a n} \dd  \a .  \end{equation}
   By using characteristic functions and Fourier inversion formula, we  deduce from  (\ref{qmn}),
  \begin{eqnarray}  \label{link}   q_m(n)& = &  
  e^{\s n}\Big( \prod_{j=m}^n  ( {1+e^{-\s j}})\Big)\P\{Y=n\} .  \end{eqnarray}
Choosing $\s$ as  the (unique) solution of the equation
 $ \sum_{j=m}^n \frac{j}{1+e^{\s j}}= n  $ gives $\P\{Y=n\}= \P\{ \overline Y=0\}$ where $\overline Y=Y-\E Y$. But here we have $\E e^{t_0\sqrt{|X_i-\E
X_i|}}\approx e^{t'_0
\sqrt j}$. Hence 
$c_2\approx e^{t'_0 \sqrt n}$, and so $c_3\gg \sqrt n$.  
  Freiman and Pitman lacked a result of this kind, and in place,   directly estimated the integral   in       (\ref{qmn})  in a 
 painstriking  work. 
 
\subsection{\gsec Weighted i.i.d.\!\! sums}
 \label{bebi1}
     The requirement on the random variables to take values in a common lattice   is generally no longer satisfied
when replacing
$X_j$ by $w_j X_j$, where $  w_j,j=1,\ldots,n$ are  real   numbers.
This occurs if  $X_j=  w_j\b_j$, where $\b_j$ is a
Bernoulli random variable and   $w_j$ are distinct integers having greatest common divisor $d$.  In this case, $\P\{X_j
\in\mathcal L(0,w_j)\}=1$ for each $j$, but   one cannot select a smaller {\it common} span (e.g. $D=d$)  since  condition
(\ref{basber}) (see also (\ref{basber1})) would be   no longer   fulfilled. This example in turn covers   important classes of
independent random variables used as probabilistic models in arithmetic. See \cite{F},\cite{FP},\cite{Po}. However,  the   representation given in Lemma \ref{lemd} extends to weighted sums. Set for $m=1,\ldots,n$,
$$S _m =\sum_{j=1 }^{ m}  w_j X_{ j}, \qq  W_m =\sum_{j=1}^m  w_j V_j,\qq M_m=\sum_{j=1}^m
w_j\e_jL_j,
\qq B_m=\sum_{j=1}^m
 \e_j .$$
 A direct consequence of  (\ref{dec0}) is
  \begin{lemma}   We have  the
representation
$$ \{S_m, 1\le m\le n\}\buildrel{\mathcal D}\over{ =}  \{ W_m  +  DM_m, 1\le m\le n\} .$$
 And, conditionally to the
$\s$-algebra generated by the sequence   $\{(V_j,\e_j), j=1, \ldots, n\}$,    $M_n$ is   a  weighted  Bernoulli random walk.
\end{lemma}
 {\gsec Problem II}   Show an approximate form of the local limit theorem for weighted i.i.d.\!\! sums.
      
{


\baselineskip 9pt
\begin{thebibliography}{99}
 \bibitem{AGKW} {Aizenmann M., Germinet F., Klein A., Warzel S.}  (2009)  {\sl On Bernoulli decompositions for random variables, concentration bounds,
and spectral localization},  Prob. Th. Rel. Fields  {\bf 143}, 219--238.
\bibitem{CFS}  Chen L.H.Y, Fang X., Shao Q--M. (2013) {\sl From Stein identities
 to moderate deviations},
 Ann. Probab.  {\bf 41}, 1, 1--443.
 \bibitem{CT} Chow Y. S., Teicher H. (2003) {\sl Probability Theory: Independence, Interchangability, Martingales}, Third Edition,
Springer Texts in Statistics, Springer-Verlag New York-Berlin-Heidelberg.
   \bibitem{MD}  {Davis B., MacDonald D.}  (1995)  {\sl An elementary proof of the local central   limit theorem},   J.
Theoretical Prob. {\bf 8} no. 3, 695--701.
 \bibitem{D} {Doney R. A.}  (1997)  {\sl One-sided local large deviation and renewal theorems in the case of infinite mean},  Prob. Th.
Rel. fields  {\bf 107}, 461--465.
 \bibitem{F} {Freiman G. A.}  (1999)  {\sl Structure theory of set addition},  Ast\'erisque  {\bf
258}, 1--33.
  \bibitem{FP}   Freiman G. A., Pitman J. (1994)
 {\sl Partitions into distinct large parts},  J. Austral. Math. Soc. (Series A)  {\bf
57}, 386--416.
   \bibitem{G}    Gamkrelidze N. G.   (1988)     {\sl On the application of a smoothness function in proving a local limit theorem},
  Theory of Prob. \& Its Appl. {\bf 33}, 352--355.
   \bibitem{G1}  {Gnedenko B. V.}  (1948)     {\sl On a local limit theorem in probability theory}, Uspekhi Mat. Nauk. {\bf III.  3}
(25), 187--190.
 \bibitem{IL} {Ibragimov  I. A.,    Linnik Y. V.}  (1971)  {\sl Independent and Stationary Sequences of Random Variables},
Wolters-Noordhoff Publishing Groningen, The Netherlands.
 \bibitem{K} {Kolmogorov M. A.}  (1958)  {\sl Sur les propri\'et\'es de fonctions
de concentrations de M. P. L\'evy},   Annales de l'I. H. P.  {\bf 16} no. 1, 27--34.
 \bibitem{di}  MacDiarmid  C.  (1998) {\sl Concentration}, Prob. Methods for Algorithmic Discrete Math., 195--248, Algorithms Combin. {\bf
16}, Springer, Berlin.
 \bibitem{M}  {MacDonald D.}  (1979)  {\sl On local limit theorems for  integer valued
random variables},   Theor. of Prob. Appl.  {\bf 33}, 352--355.
\bibitem{M1} {MacDonald D.}  (1979)  {\sl A local limit theorem for large deviations of sums of independent, non-identically distributed
random variables},   Annals of Prob. {\bf 7} no. 3, 526--531.
\bibitem{Mat} Matskyavichyus V. K. {\sl On a lower bound for the convergence rate in a local limit theorem}, Theor. Prob. \& Appl.,
810--814.
 \bibitem{Mi}    Mitrinovi\'c D. S. (1970)   {\sl  Analytic inequalities}, Springer Verlag {\bf 165}.
  \bibitem{Mu}     Mukhin A. B.  (1991)     {\sl Local limit theorems for lattice random variables},   Theory of Prob. \& Its Appl. {\bf
36} no 4, 698--713.
\bibitem{Mu1} Mukhin A. B. (1984) {\sl Local limit theorems for distributions of sums of independent random
vectors}, Theory of Prob. \&  Its Appl. {\bf 29}, 369--375.
\bibitem{[P]}      Petrov, V. V. (1975)     {\em Sums
of Independent Random Variables}, Ergebnisse der Math. und ihre
 Grenzgebiete  {\bf 82}, Springer.
  \bibitem{Po}    Postnikov A. G.  (1988)    {\sl Introduction to analytic number theory}, AMS Translation of mathematical monographs  {\bf
68}. First publ. in Russian in 1971.
  \bibitem{Pr}    Prohorov Y. V.  (1954)  {\sl On a local limit theorem for lattice distributions}, (Russian)  Dokl. Akad. Nauk. SSSR (N.S.)
{\bf 98}, 535--538.
   \bibitem{Ro}    Rozanov Y. A.  (1957)     {\sl On a local limit theorem for lattice distributions},
  Theory of Prob. \& Its Appl. {\bf
2} no 2, 260-265.
    \bibitem{W}  {Weber M.} (2011)   {\sl  A sharp correlation inequality   with an
application to almost sure   local
 limit theorem},     Probability and Math. Stat.  {\bf 31}, Fasc. 1,   79--98.
    \end{thebibliography}
  \end{document}